\documentclass[12pt,a4paper]{amsart}
\usepackage{xcolor,amsmath,amssymb,latexsym,amsfonts,amsthm,graphicx,hyperref}
\usepackage[parfill]{parskip}

\usepackage{tikz-cd}

\definecolor{darkgreen}{RGB}{0,100,0}

\newtheorem{theorem}{Theorem}
\newtheorem{corollary}[theorem]{Corollary}

\newtheorem{lemma}[theorem]{Lemma}
\newtheorem{proposition}[theorem]{Proposition}

\def\bds{\begin{displaystyle}}
\def\eds{\end{displaystyle}}

\def\1{{\mathchoice {1\mskip-4mu\mathrm l}      
{1\mskip-4mu\mathrm l}
{1\mskip-4.5mu\mathrm l} {1\mskip-5mu\mathrm l}}}
\newcommand{\floor}[1]{\left\lfloor #1 \right\rfloor}

\usepackage{enumerate}

\usepackage{stmaryrd}
\usepackage{mathtools}
\DeclarePairedDelimiter\IeFer{\llbracket}{\rrbracket} 
\DeclarePairedDelimiter\abs{\lvert}{\rvert} 
\DeclarePairedDelimiter\pa{(}{)} 
\DeclarePairedDelimiter\acc{\{}{\}} 
\DeclarePairedDelimiter\cro{[}{]} 
\newcommand{\eps}{\varepsilon}
\newcommand{\sub}[1]{2^{\IeFer{#1}}} 

\newcommand{\E}{\mathbb{E}}
\renewcommand{\P}{\mathbb{P}}
\newcommand{\sE}{\mathcal{E}}

\newcommand{\s}{\,|\,} 

\newcommand{\st}{\,:\,} 

\newcommand{\N}{\mathbb{N}}
\newcommand{\R}{\mathbb{R}}

\newcommand{\Br}{\mathcal B}
\newcommand{\Er}{\mathcal E}

\newcommand{\Z}{\mathbb{Z}}
\newcommand{\Nt}{\widetilde{N}}

\newcommand{\Abeta}{\text{\lparen A}_\beta\text{\rparen}}

\newcommand{\Cbeta}{\text{\lparen C}_\beta\text{\rparen}}
\newcommand{\modulo}{\!\!\!\mod}
\newcommand{\limites}[2]{\overset{#1}{\underset{#2}{\longrightarrow}}}

\newcommand{\F}{\mathbb F}

\newcommand{\phit}{\widetilde \phi}
\newcommand{\gt}{\widetilde g}
\newcommand{\et}{\widetilde e}
\newcommand{\xt}{\widetilde x}
\newcommand{\phih}{\widehat \phi}
\newcommand{\Ph}{\widehat P}
\newcommand{\eh}{\widehat e}

\newcommand{\ind}[1]{\mathbf{1}_{#1}}

\begin{document}

\title{Are random random walks normal?}
\author[K. Hamza]{Kais Hamza}
\address{School of Mathematics, Monash University, VIC 3800, Australia}
\email{Kais.Hamza@monash.edu}
\author[L. Tournier]{Laurent Tournier}
\address{Université Sorbonne Paris Nord, LAGA, CNRS, UMR 7539, F-93430, Villetaneuse, France}
\email{tournier@math.univ-paris13.fr}

\begin{abstract}
Given a symmetric simple random walk $(X_n)_{n \ge 0}$, the family of all symmetric simple random walks $(Y_n)_{n \ge 0}$ adapted to the filtration of $(X_n)_{n \ge 0}$ was studied in Collevecchio et al.\ (2022). In particular, the authors established necessary and sufficient conditions under which the suitably normalized two-dimensional process $((X_n, Y_n))_{n \ge 0}$ converges  weakly to a two-dimensional Brownian motion. 

When this occurs, we say that the random walk $(Y_n)_n$ is \emph{normal}  (with respect to $(X_n)_n$). In this paper, we investigate whether a ``randomly selected'' $(Y_n)_{n \ge 0}$ is normal. We consider a very general randomization procedure and look at both the quenched and annealed settings.
\end{abstract}

\maketitle

\section{Introduction}

The interplay between dependence, randomness, and structural symmetry in discrete stochastic systems continues to provide rich insights into the theory of random walks and their continuous limits. The framework of \emph{General Bootstrap Random Walks (GBRWs)}, introduced by Collevecchio, Hamza, Shi, and Williams (2021), considers a model in which the increments of a symmetric simple random walk $X$ are recycled through non-anticipative, measure-preserving transformations, to define the increments of a new symmetric simple random walk $Y$. 
These transformations, deterministic in the original formulation, give rise to a two-dimensional process
$$
(X_n, Y_n), \qquad n \ge 0,
$$
whose asymptotic behavior depends on the specific recycling rule governing the evolution of $Y_n$.

To describe the dynamics precisely, let $(\xi_n)_{n\ge1}$ denote the increments of $(X_n)_{n\ge0}$; that is, a sequence of independent and identically distributed random variables taking values in $\sE=\{-1,+1\}$ with
$$
\mathbb{P}(\xi_n = +1) = \mathbb{P}(\xi_n = -1) = \tfrac12.
$$
The second coordinate is constructed by recycling the increments of $(X_n)_{n\ge1}$ through a non-anticipative, measure-preserving transformation. We set $\eta_1 = \xi_1$ and, for each $n \ge 2$, define
$$
\eta_n = \phi_{n-1}(\xi_1,\dots,\xi_{n-1})\,\xi_n,
$$
where $\phi_{n-1}$ is deterministic and takes values in $\sE$. We then define
$$
Y_0 = 0, \qquad\text{and for $n\ge1$,}
\quad 
Y_n = \sum_{k=1}^n \eta_k .
$$

The predictability of $(\phi_{n-1})_n$ ensures that $(\eta_n)_n$ is non-anticipative, while the multiplicative structure guarantees that the increments satisfy
$$
\forall (y_1,\ldots,y_n)\in\sE^n,\quad\mathbb{P}(\eta_1=y_1,\ldots,\eta_n=y_n) = \tfrac1{2^n},
$$
so that $Y$ has the same increment distribution as $X$. 

The earlier work~\cite{CHSW} established conversely that these properties only hold for such transformations, and developed a hierarchy of limit theorems for the normalized pair
$$
W_t^{(n)} = \left(\frac{X_{\lfloor nt \rfloor}}{\sqrt{n}}, \frac{Y_{\lfloor nt \rfloor}}{\sqrt{n}}\right),
$$
based on necessary and sufficient conditions for convergence to Gaussian limits --- either independent, correlated, or degenerate --- summarized in the following result. 

\begin{theorem}[\cite{CHSW}]\label{mainth}
$W^{(n)}$ converges weakly to a two-dimensional Brownian motion (possibly degenerate) with correlation $\rho$ if and only if
\begin{description}
  \item[(A)] $\displaystyle \rho = \lim_{n \to \infty} \frac{1}{n} \sum_{k=1}^n \mathbb{E}[\zeta_{k-1}]$ exists, where $\zeta_{k-1} = \eta_k \xi_k(=\phi_{k-1}(\xi_1,\ldots,\xi_{k-1})$) ($\zeta_0=1$);
  \item[(B)] $\displaystyle \lim_{n \to \infty} \frac{1}{n^2} \sum_{k=1}^n \sum_{\ell=1}^n \mathbb{E}[\zeta_{k-1}\zeta_{\ell-1}] = \rho^2.$
\end{description}
\end{theorem}

These conditions ensure convergence of the quadratic covariation process of $(X_n, Y_n)$ and completely describe the Gaussian limit regime.

Note that (A) and (B) together are equivalent to (see Theorem VIII.3.11 of \cite{JacodShiryaev})
\begin{description}
    \item[(C)] $\displaystyle \frac1n\sum_{k=1}^n\zeta_{k-1}\limites{}{n\to\infty}\rho$ in probability.
\end{description}

Several prototypical examples from~\cite{CHSW} illustrate how different recycling rules generate distinct limit behaviors:
\begin{itemize}
  \item \emph{Product bootstrap:} when $\eta_n = \prod_{k=1}^n \xi_k$, despite the strong dependence, a normalized $(X_n, Y_n)$ converges to a pair of independent Brownian motions.
  \item \emph{Correlated Brownian motion:} when $\eta_n = \max(\xi_{n-m}, \ldots, \xi_{n-1})\xi_n$ for fixed $m \ge 1$, the limit process is a two-dimensional Brownian motion with correlation $\rho = 1 - 2^{1-m}$, which is arbitrarily close to $1$ for large values of $m$.
  \item \emph{Sign transformation:} $\eta_n = \operatorname{sgn}(X_{n-1})\xi_n$ corresponds to a discrete analogue of the Lévy transformation
  \[
  (B_t)_{t\ge0} \mapsto \pa[\bigg]{\int_0^t \operatorname{sgn}(B_s)\,dB_s}_{t\ge0},
  \]
  it fails to satisfy (B)
  and leads to a \emph{non-Gaussian} limiting process
  \[
  W_t = \left(B_t, \int_0^t \operatorname{sgn}(B_s)\,dB_s\right).
  \]
\end{itemize}

The article \cite{CHSW} relied on a parameterization of the aforementioned non-anticipative, measure preserving transformations (i.e.\ ``recycling mechanisms'') by a family of binary coefficients $\beta_{n,K}\in\{0,1\}$, indexed by subsets $K\subset\IeFer n=\{1,\ldots,n\}$, through the representation
\begin{equation}
\eta_n = \phi_{n-1}(\xi_1,\ldots,\xi_{n-1})\xi_n = \xi_n\prod_{K \subset \IeFer{n-1}} \xi_{[K]}^{\,\beta_{n-1,K}},\label{eq:representation_historic}
\end{equation}
where $\xi_{[K]} = \max_{k \in K} \xi_k$. 
Every such transformation indeed admits a unique representation of this form (this shall also be discussed below, see in particular Appendix~\ref{app:boolean_functions}). 
Note that the quantities $\xi_{[K]}$ simply encode the presence of positive entries on subsets of coordinates; indeed, $\xi_{[K]}=1\ \Leftrightarrow\ \exists i\in K,\,\xi_i=1$. 

\medskip
\noindent

The present paper extends this framework by introducing \emph{randomness} into the bootstrap transformation. In the original GBRW, the recycling rules were deterministic. Here, we allow the recycling functions $\phi_n$ themselves to be random, thereby introducing an additional layer of stochasticity that interacts with the intrinsic randomness of the underlying walk. This naturally leads to annealed and quenched formulations of the bootstrap process: in the annealed setting, one averages over both the evolution of the walk and the randomness of the resampling mechanism, whereas in the quenched setting, the random environment is fixed and one studies the walk conditionally on the realization of the recycling mechanism.

Allowing the sequence $(\phi_n)_n$ to be random raises immediate questions regarding both its marginal distributions and its dependence structure. A natural starting point is to assume that the $\phi_n$ are independent and uniformly distributed over a suitable class, and independent of the underlying walk. However, our aim is to develop a more flexible framework that accommodates richer forms of dependence.

The space of functions $\phi_n:\mathcal{E}^n\to\mathcal{E}$ forms a $2^n$-dimensional $\mathbb{F}_2$-vector space ($\mathbb{F}_2=\mathbb{Z}/2\mathbb{Z}=\{0,1\}$), naturally identified with $\{0,1\}^{2^n}$. Consequently, specifying a probability distribution on a random function $\phi_n$ is equivalent to specifying a joint distribution for its coordinates $(\beta_{n,a})_{a\in\mathcal{E}^n}$ with respect to a chosen basis $\mathfrak{g}_n=(g_{n,a})_{a\in\mathcal{E}^n}$ of $\mathcal{E}^{\mathcal{E}^n}$: for any $\mathfrak{b}_n=(b_{n,a})_{a\in\mathcal{E}^n}$, if we define the function $f_n:x\mapsto\prod_{a\in\mathcal{E}^n}g_{n,a}(x)^{b_{n,a}}$, then
$$\mathbb{P}(\phi_n=f_n) = \mathbb{P}(\beta_{n,a}=b_{n,a},\forall a\in\mathcal{E}^n).$$

The case of independent uniformly distributed functions corresponds to choosing the coefficients $(\beta_{n,a})$ as independent and identically distributed $\mathrm{Bernoulli}(1/2)$ random variables. In general however, the resulting distribution on functions depends on the choice of bases, which is thus a crucial component of the construction. 

One may expect the construction to address the temporal structure of the model, i.e.\ its dependence on the parameter $n$. Thus, as discussed in Appendix~\ref{sec:consistency}, it is natural to impose consistency conditions on the sequence of bases. Imposing these conditions severely restricts the admissible constructions: there exist only \emph{two} consistent sequences of bases. Each arises from a partial order on $2^{\IeFer{n}}$, hence the indexing shall thereafter be transferred from $\mathcal{E}^n$ to $2^{\IeFer{n}}$ for convenience.
In each setting, the corresponding basis is denoted by
$\mathfrak{g}_n=(g_{n,K})_{K\subset \IeFer n}$
with the distinction between the two cases made clear by the surrounding context. Alternatively, we argue in Appendix~\ref{sec:representations} that these two sequences of bases are two natural examples, arising as canonical bases and ubiquitous in the literature on binary functions. 

These two sequences of bases are referred to as the \emph{totally unordered case} (T.U.\ case) and the \emph{inclusion-induced case} (I.I.\ case).
Let, for $x\in\mathcal{E}^n$, $\kappa(x)=\kappa_n(x)=\{k\in\IeFer{n}:x_k=-1\}$. 

\begin{itemize}
\item In the totally unordered case, one has $g_{n,K}(x)=-1$ if and only if $\kappa(x)=K$, so that
$$\forall x\in\Er^n,\qquad\phi_n(x)=(-1)^{\beta_{n,\kappa(x)}}.$$
Hence, $\phi_n(x)$ is determined directly by $\beta_{n,\kappa(x)}$.

\item In the inclusion-induced case, corresponding to the representation~\eqref{eq:representation_historic}, $g_{n,K}(x)=x_{[K]}(=\max_{k\in K}x_k)$, so that
$$\forall x\in\Er^n,\qquad\phi_n(x)
=\prod_{K\subset\IeFer{n}} x_{[K]}^{\beta_{n,K}}
=(-1)^{\sum_{K\subset\kappa(x)} \beta_{n,K}}.$$
Therefore, $\phi_n(x)$ is determined by the parity of $\displaystyle\sum_{K\subset\kappa(x)} \beta_{n,K}$ and hence depends on all variables $\beta_{n,K}$ corresponding to subsets $K$ of $\kappa(x)$. Although this won't be used in the following, one may note that the coefficients $\beta_{n,K}$ can conversely be explicitly computed from the values of $\phi_n$ through a dual formula, see~\eqref{eq:inversion}.
\end{itemize}

The aim of this paper is to investigate a broad class of randomizations arising from these constructions. Each sequence of bases induces a distinct mechanism of randomization, and we study both under a range of assumptions on the law $\mu$ of the family of random coefficients $(\beta_{n,K})_{K\subset\IeFer{n},n\in\mathbb{N}}$.

It is natural to expect that the case in which the coordinates are independent and identically distributed Bernoulli$(1/2)$ random variables -- equivalently, when the $\phi_n$ are independent and uniformly distributed -- gives rise to a standard Brownian motion. As the following sections demonstrate, this is indeed the case. This setting will appear as a special instance of the more general framework developed below.

Rather than assuming full independence, we consider a setting in which the family
$\bigl((\beta_{n,K})_{K\subset\IeFer{n}}\bigr)_{n\in\mathbb{N}}$
is partitioned into independent blocks, allowing for various dependence structures within each block. We denote by $p_{n,K}$ the parameter of the Bernoulli variable $\beta_{n,K}$. 

We consider three natural ways of partitioning the coordinates: according to the level $n$, according to the subset $K$, and according to $|K|$, the size of $K$. 

In each of the considered cases, we show that $(W^{(n)})_n$ converges weakly to a nondegenerate Brownian motion, with either independent or correlated components, either under the annealed distribution or a.s.\ under the quenched distribution. The following summarizes the main results of this paper. Note that when quenched convergence holds, then annealed convergence holds as a consequence and is not mentioned below; when annealed convergence is mentioned, its quenched counterpart is unresolved in this paper and remains as an open problem.  

\begin{enumerate}

\item The blocks are indexed by $n$: the blocks $\mathfrak b_n=(\beta_{n,K})_{K\subset\IeFer n}$, $n\in\mathbb{N}$, are independent.
\begin{center}
\begin{tabular}{|c|c|c|c|c|}\hline
Basis & Dependence inside & Parameters & Convergence & Statement\\\hline\hline
T.U. & Arbitrary & $p_{n,K}=p_n$, $\frac1n\sum_{k=1}^n p_k\to p$ & Quenched & Proposition \ref{tot_unord_quenched}\\\hline
\end{tabular}
\end{center}

\item The blocks are indexed by $K$: the blocks $\mathfrak{b}_K = (\beta_{n,K})_{n \ge \max K}$, $K\subset\mathbb{N}$, $|K|<+\infty$, are independent.
\begin{center}

\begin{tabular}{|c|c|c|c|c|}\hline
Basis & Dependence inside & Parameters & Convergence & Statement\\\hline\hline
I.I. & Arbitrary & $\eps<p_{n,K}<1-\eps$  &  Quenched & Theorem \ref{th:Kindep}\\\hline
T.U. & Equal & $p_{n,K}=p$ & Quenched & Proposition \ref{tot_unorder_Kindep}\\\hline
\end{tabular}
\end{center}

\item The blocks are indexed by $n$ and $k=|K|$: the blocks $\mathfrak{b}_{n,k} = (\beta_{n,K})_{|K|=k}$, $n\in\mathbb{N}$, $0\le k\le n$, are independent.
\begin{center}

\begin{tabular}{|c|c|c|c|c|}\hline
Basis & Dependence inside & Parameters & Convergence & Statement\\\hline\hline
I.I. & Equal & $p_{n,K}=p_{n,\abs K}$,  $\eps<p_{n,k}<1-\eps$ & Annealed & Corollary \ref{cor:Markov}\\\hline
\end{tabular}
\end{center}

\item The blocks are indexed by $k=|K|$: the blocks $\mathfrak{b}_k = (\beta_{n,K})_{n\ge k\vee1,|K|=k}$, $k\in\mathbb{N}\cup\{0\}$, are independent.
\begin{center}

\begin{tabular}{|c|c|c|c|c|}\hline
Basis & Dependence inside & Parameters & Convergence & Statement\\\hline\hline
I.I. & Equal & $p_{n,K}=p_{\abs K}$, $\eps<p_{k}<1-\eps$ & Annealed & Corollary \ref{cor:Markov-Depend}\\\hline
\end{tabular}
\end{center}

\end{enumerate}

\subsection*{Organization of the paper} After a formal introduction of the setting (Section~\ref{sec:setting}), we focus on the I.I.\ representation, first under a general assumption on the law $\mu$ of $(\beta_{n,K})_{n\in\N,K\subset\IeFer n}$, that covers the i.i.d.\ case (Section~\ref{sec:general}), then under a more restrictive case that we interpret as 1-factor Markov representations (Section~\ref{sec:II_1fMarkov}). The results for the T.U.\ representation, which turns out to enforce more independence and enables simpler proofs (yet under different assumptions on $\mu$), are gathered in Section~\ref{sec:totally_unordered}. Finally, the relevance of the choice of representations is discussed in Appendix~\ref{app:boolean_functions}, where it is justified first as being a consequence of a natural consistency condition, second as being two instances of canonical, hence natural, bases. 

\section{Formal setting}\label{sec:setting}

The setting can be formalized as follows. 

Let $(\mathfrak{g}_n)_{n\in\N}=((g_{n,K})_{K\subset \IeFer n})_{n\in\N}$ be either of the following two sequences of bases: 
\begin{description}
    \item[I.I.\ case] for all $n\ge1$, $K\subset\IeFer n$, $g_{n,K}:x\in\mathcal E^n\mapsto x_{[K]} := \max_{k\in K}x_k\in\mathcal E$; 
    \item[T.U.\ case] for all $n\ge1$, $K\subset\IeFer n$,  $g_{n,K}:x\in\mathcal E^n\mapsto (-1)^{\ind{\{K=\kappa_n(x)\}}}\in\mathcal E$, 
\end{description}
where we recall that $\kappa_n(x)=\acc{k\in\IeFer n\st x_k=-1}$ for $x\in\Er^n$.

On the probability space $(\Omega,\mathcal F, \mathbb P)$, the $\sE$-valued sequence $(\xi_n)_{n\ge1}$ and the $\{0,1\}$-valued sequence $\beta:=(\beta_{n,K})_{n\ge0,\ K\in\sub n}$ are independent sequences such that
\begin{itemize}
    \item $\xi:=(\xi_n)_{n\ge1}$ are i.i.d.\ random variables uniformly distributed in $\sE$;
    \item $\beta:=(\beta_{n,K})_{n\in\mathbb N, K\in\sub n}$ follows some distribution $\mu$ (to be specified in the statements of results).
\end{itemize}
Define $\zeta_0=1$ and, for all $n\ge1$, 
\[\zeta_n:=\prod_{K\subset\IeFer n}g_{n,K}(\xi_1,\ldots,\xi_n)^{\beta_{n,K}}\quad \mbox{ and }\quad \eta_n:=\zeta_{n-1}\xi_n.\]
That is, 
\begin{itemize}
    \item in the I.I.\ case, 
\begin{equation}\label{eq:expression_zeta}
\zeta_n = \prod_{K \subset \IeFer{n}} \xi_{[K]}^{\beta_{n,K}} =
(-1)^{\sum_{K\subset\kappa_n(\xi)}\beta_{n,K}} ;
\end{equation}
    \item in the T.U.\ case, $\zeta_n =  (-1)^{\beta_{n,\kappa_n(\xi)}}.$
\end{itemize}
Finally, $X_n=\xi_1+\cdots+\xi_n$ and $Y_n=\eta_1+\cdots+\eta_n$ for $n\ge 0$. And, for $n\ge1$ and $t\ge0$, $W^{(n)}_t=\frac1{\sqrt n}(X_{\floor{nt}},Y_{\floor{nt}})$. 

For any sequence $\beta$, denote by $P^\beta$ the conditional distribution of $\zeta$ given $\beta$. Thus, \[\mathbb P(\zeta\in\cdot)=\int P^\beta(\zeta\in\cdot)d\mu(\beta).\]

Recall that by construction, for any $\beta$, under $P^\beta$, the sequence $\eta$ is distributed as $\xi$. Since this distribution does not depend on $\beta$, the same therefore holds under $\mathbb P$ as well. 

We are interested in the asymptotics of the distribution of the variable $W^{(n)}$ under $\mathbb P$ (known as the \emph{annealed} or \emph{averaged} measure), and under $P^\beta$ for a given $\beta$ (known as the \emph{quenched} measure). 

Because of Theorem~\ref{mainth}, $W^{(n)}$ converges under $P^\beta$ to a two-dimensional Brownian motion with correlation $\rho_\beta$ if, and only if one has both conditions
\begin{description}
    \item[$\text{(A}_\beta\text{)}$] $\displaystyle\frac1N\sum_{n=1}^{N}E^\beta[\zeta_{n-1}]\to\rho_\beta$ as $N\to\infty$;  
    \item[$\text{(B}_\beta\text{)}$] $\displaystyle\frac1{N^2}\sum_{m=1}^N\sum_{n=1}^{N}E^\beta[\zeta_{m-1}\zeta_{n-1}]\to\rho_\beta^2$ as $N\to\infty$. 
\end{description}

Furthermore, as can be be checked from the proof of Theorem~\ref{mainth}, its statement holds unchanged under the annealed measure (indeed $(X_n,Y_n)$ is still a martingale under $\P$): $W^{(n)}$ converges under $\P$ to a two-dimensional Brownian motion with correlation $\rho$ if, and only if (A) and (B) hold.

\section{I.I.\ representation, main result}\label{sec:general}

Let us consider representations in the I.I.\ bases. 

Let $p\in(0,1)$ and consider the following two sets of assumptions on the family $(\beta_{n,K})_{n,K}$. 

\begin{description}
    \item[(G1)] The family $(\beta_{n,K})_{n\in\mathbb{N},\,K\subset\IeFer n}$ consists of i.i.d.\ Bernoulli random variables with parameter~$p$.

    \item[(G2)] For all $n\in\mathbb{N}$ and $K\subset\IeFer n$, we set $\beta_{n,K}=\beta_K$, where $(\beta_K)_{K\subset\mathbb{N},\, K\text{ finite}}$ is a family of i.i.d.\ Bernoulli random variables with parameter $p$.
\end{description}

These are two particular (and somehow extreme) cases of the general framework of our main result that we now introduce. 

For each finite subset $K \subset \mathbb{N}$, let $\mathfrak{b}_K = (\beta_{n,K})_{n \ge \max K}$. 

In this section, we impose independence of the sequence $(\mathfrak{b}_K)_{K}$ while allowing arbitrary (possibly non-stationary) dependence within each sequence $\mathfrak{b}_K$. More specifically, we make the following main assumption. 

\begin{description}
    \item[(G)] The family $(\mathfrak{b}_K)_{K \subset \mathbb{N},\, |K| < \infty}$ is mutually independent, the random variables $\beta_{n,K}$ are Bernoulli random variables, and their parameters $p_{n,K}$ are elliptic, i.e.\ uniformly bounded away from $0$ and $1$: there exists $\varepsilon\in(0,1/2)$ such that
    \[
    p_{n,K}\in[\varepsilon,1-\varepsilon]
    \]
    for all admissible $(n,K)$.
\end{description}

Our result in this setting is the following. 

\begin{theorem}[I.I.\ representation]\label{th:Kindep} Assume (G). Then convergence to a standard Brownian motion (with zero correlation) holds for almost every $\beta$: ``almost every random walk is normal''. 
\end{theorem}

\begin{corollary}[I.I.\ representation] Assume (G). Then convergence to a standard Brownian motion (with zero correlation) holds under the annealed measure $\P$: ``the random random walk is normal''. 
\end{corollary}

We split the proof of Conditions~(A) and (B), leading to Theorem~\ref{th:Kindep}, in two propositions. 

\begin{proposition}[I.I.\ representation] Assume (G). Then, for almost every $\beta$, condition $\text{(A}_\beta\text{)}$ holds with $\rho_\beta=0$. 
\end{proposition}

\begin{proof}
In order to prove $\text{(A}_\beta\text{)}$, we prove the following stronger statement: 
\begin{equation}\label{eq:SC-A_beta}
\sum_{n=0}^\infty\E[E^\beta[\zeta_n]^2]<\infty. 
\end{equation}

Indeed, the above summability implies by monotone convergence that the random sum $\sum_{n=0}^\infty E^\beta[\zeta_n]^2$ is integrable, hence almost surely finite, which implies that $E^\beta[\zeta_n]$ converges to 0 almost surely, which in turn by Cesàro's lemma implies $\text{(A}_\beta\text{)}$ with $\rho_\beta=0$. 

To prove \eqref{eq:SC-A_beta}, let $\zeta'_n$ be a copy of $\zeta_n$ defined on the same probability space such that $\zeta_n$ and $\zeta'_n$ are independent, conditionally on $\beta=(\beta_{n,K})_{n\in\mathbb{N},\,K\subset\IeFer n}$. Then, we simply write $\E[E^\beta[\zeta_n]^2]=\E[E^\beta[\zeta_n\zeta'_n]]$. More precisely, we may consider that, up to enlarging $(\Omega,\mathcal F, \P)$ we have a random sequence $\xi'=(\xi'_n)_{n\ge1}$ that is an independent copy of $\xi$, which enables to define $\zeta'_n$ for all $n\ge0$ by $\zeta'_0=1$ and, for all $n\ge1$, 
\[\zeta'_n:=\prod_{K\subset\IeFer n}(\xi'_{[K]})^{\beta_{n,K}}.\]
The notation $P^\beta$ from now on denotes the conditional distribution of $(\zeta,\zeta')$ given $\beta$. 

Denoting by $A_n(\xi)=2^{\kappa_n(\xi)}$ the power set of $\kappa_n(\xi)$, we have, from the expression~\eqref{eq:expression_zeta} of $\zeta_n$ in the I.I.\ case,
\begin{align*}
    E^\beta[\zeta_n]^2 = E^\beta[\zeta_n\zeta'_n] = E^\beta[(-1)^{\sum_{K\in A_n(\xi)\Delta A_n(\xi')}\beta_{n,K}}],
\end{align*}
where $\Delta$ denotes the symmetric difference. 
Then, by Fubini's theorem, 
\begin{align}
    \E[E^\beta[\zeta_n]^2]
        & = \E\cro[\bigg]{\E\cro[\Big]{(-1)^{\sum_{K\in A_n(\xi)\Delta A_n(\xi')}\beta_{n,K}} \Big| \xi,\xi' }}.
\end{align}
Given $\xi,\xi'$, the exponent is a sum of independent Bernoulli r.v.'s. As a conclusion,
\begin{align*}
    \E[E^\beta[\zeta_n]^2]
        & = \E\cro[\bigg]{\prod_{K\in A_n(\xi)\Delta A_n(\xi')} (1-2p_{n,K})}
        \le \E\cro[\bigg]{\prod_{K\in A_n(\xi)\Delta A_n(\xi')} \abs{1-2p_{n,K}}}\\
        & \le \E\cro[\bigg]{(1-2\epsilon)^{\#\pa{A_n(\xi)\Delta A_n(\xi')}}} .
\end{align*}
Observe that, clearly, 
\begin{align*}
A_n(\xi)\Delta A_n(\xi') \supset A_n(\xi)\cap A_n(\xi')^c 
& = \{K\subset\IeFer n \st \forall i\in K,\ \xi_i=-1\text{ and }\exists j\in K, \xi'_j=1\}\\
& \supset\{K\subset\IeFer n\st \forall i\in K, (\xi_i=-1\text{ and } \xi'_i=1)\} =  2^{\{i\in\IeFer n\st \xi_i=-1,\ \xi'_i=1\}} 
\end{align*}
hence, with $z_i=\ind{(\xi_i=-1)}$ and similarly for the sequence $\xi'$,
\[\#\pa[\big]{A_n(\xi)\Delta A_n(\xi')}\ge 2^{\sum_{i=1}^n z_i(1-z'_i)}, \]
and the above sum is binomially distributed with parameters $n$ and $1/4$. Finally, 
\begin{align*}
    \E[E^\beta[\zeta_n]^2]
        & \le \sum_{k=0}^n {n\choose k}\pa[\bigg]{\frac14}^k\pa[\bigg]{\frac34}^{n-k}(1-2\epsilon)^{2^k}\\
        & \le \pa[\bigg]{\frac34}^n \sum_{k=0}^n {n\choose k}(1-2\epsilon)^{2^k} \le a^n
\end{align*}
where the last inequality holds for any fixed $a\in(\frac34,1)$, for large $n$, due to Lemma~\ref{lem:subexponential}. This bound ensures the summability of $\E[E^\beta[\zeta_n]^2]$, which concludes the proof. 
\end{proof}

\begin{proposition}[I.I.\ representation] Assume (G). \label{prop:B_quenched}
Then Condition $(B_\beta)$ holds for almost every $\beta$, with $\rho_\beta=0$.  
\end{proposition}

\begin{proof}
    First, with the same notations as in the previous proof, for any integers $m, n$, 
\begin{align}
    \E[E^\beta[\zeta_m\zeta_n]^2]
        & = \E[E^\beta[\zeta_m\zeta_n]E^\beta[\zeta'_m\zeta'_n]]=\E[E^\beta[\zeta_m\zeta_n\zeta'_m\zeta'_n]]=\E\cro[\big]{\E[\zeta_m\zeta_n\zeta'_m\zeta'_n\s \xi,\xi']}
        \label{eq:zeta_produit}
\end{align}
where the last step is Fubini's theorem. 
Then we shall use identity~\eqref{eq:expression_zeta}. 

We have, for $m<n$, 
\begin{align*}
    \zeta_m\zeta_n
        & =(-1)^{\sum_{K\subset \kappa_m(\xi)}\beta_{m,K}+\sum_{K\subset \kappa_n(\xi)}\beta_{n,K}}\\
        & = (-1)^{\sum_{K\in A_m(\xi)}(\beta_{m,K}+\beta_{n,K})}(-1)^{\sum_{K\in A_{m,n}(\xi)}\beta_{n,K}}
\end{align*}
where 
\[
A_{m,n}(\xi)=2^{\kappa_n(\xi)}\cap(2^{\IeFer m})^c
\]
so that $A_n(\xi)=A_m(\xi)\uplus A_{m,n}(\xi)$.
By the previous computation and recalling the independence of the random variables $\beta_{n,K}$ across distinct sets $K$, for all $m<n$, 
\begin{align*}
    \E[E^\beta[\zeta_m\zeta_n]^2]
        & = \E\cro[\bigg]{\E\cro[\Big]{(-1)^{\sum_{K\in A_{m}(\xi)\Delta A_{m}(\xi')}(\beta_{m,K}+\beta_{n,K})}(-1)^{\sum_{K\in A_{m,n}(\xi)\Delta A_{m,n}(\xi')}\beta_{n,K}}\Big|\xi,\xi'}}\\
        & = \E\cro[\bigg]{\E\cro[\Big]{(-1)^{\sum_{K\in A_{m}(\xi)\Delta A_{m}(\xi')}(\beta_{m,K}+\beta_{n,K})}\Big|\xi,\xi'}\E\cro[\Big]{(-1)^{\sum_{K\in A_{m,n}(\xi)\Delta A_{m,n}(\xi')}\beta_{n,K}}\Big|\xi,\xi'}}\\
        & \le \E\cro[\bigg]{\abs[\bigg]{\E\cro[\Big]{(-1)^{\sum_{K\in A_{m}(\xi)\Delta A_{m}(\xi')}(\beta_{m,K}+\beta_{n,K})}\Big|\xi,\xi'}\E\cro[\Big]{(-1)^{\sum_{K\in A_{m,n}(\xi)\Delta A_{m,n}(\xi')}\beta_{n,K}}\Big|\xi,\xi'}}}\\
        & \le \E\cro[\bigg]{\E\cro[\Big]{1\Big|\xi,\xi'}\cdot\abs[\bigg]{\E\cro[\Big]{(-1)^{\sum_{K\in A_{m,n}(\xi)\Delta A_{m,n}(\xi')}\beta_{n,K}}\Big|\xi,\xi'}}}\\
        & = \E\cro[\bigg]{\abs[\bigg]{\prod_{K\in A_{m,n}(\xi)\Delta A_{m,n}(\xi')}(1-2p_{n,K})}}\\
        & \le \E\cro[\bigg]{(1-2\epsilon)^{\#\pa[]{A_{m,n}(\xi)\Delta A_{m,n}(\xi')}}}.
\end{align*}
We have the crude inclusions
\[
A_{m,n}(\xi)\Delta A_{m,n}(\xi')
    \supset A_{m,n}(\xi)\cap(A_{m,n}(\xi'))^c \supset 2^{\{i\in\IeFer n\setminus\IeFer m\st \xi_i=-1,\,\xi'_i=1\}},
\]
hence
\begin{align*}
    \#\pa[]{A_{m,n}(\xi)\Delta A_{m,n}(\xi')}
    & \ge 2^{\sum_{i=m+1}^n z_i(1-z'_i)},
\end{align*}
where as before $z_i=\ind{(\xi_i=-1)}$ and similarly for the sequence $\xi'$.
Again, the above sum is binomially distributed with parameters $n-m$ and $1/4$. Thus, 
\begin{align*}
    \E[E^\beta[\zeta_m\zeta_n]^2]
        & \le \sum_{k=0}^{n-m} {n-m\choose k}\pa[\bigg]{\frac14}^{\!\!k}\pa[\bigg]{\frac34}^{\!\!n-m-k}(1-2\epsilon)^{2^k}\\
        & \le \pa[\bigg]{\frac34}^{\!\!n-m} \sum_{k=0}^{n-m} {n-m\choose k}(1-2\epsilon)^{2^k} \le a^{n-m}
\end{align*}
where the last inequality holds for any fixed $a\in(\frac34,1)$, for large values of $n-m$, due to Lemma~\ref{lem:subexponential}. In particular, 
\begin{align*}
    \E\cro[\bigg]{\sum_{\substack{m,n\st\\\abs{m-n}\ge\sqrt n}}E^\beta[\zeta_m\zeta_n]^2}
        & = \sum_{\substack{m,n\st\\\abs{m-n}\ge\sqrt n}}\E\cro[\big]{E^\beta[\zeta_m\zeta_n]^2}
         \le \sum_{\substack{m,n\st\\\abs{m-n}\ge\sqrt n}} a^{\abs{n-m}}\\
        & \le \sum_{n=0}^\infty 2\sum_{k=\floor{\sqrt n}}^\infty a^k \le \frac 2{a(1-a)}\sum_{n=0}^\infty a^{\sqrt n}  <\infty,
\end{align*}
hence the sum in the first expectation is finite a.s., which implies that, for almost every $\beta$, $E^\beta[\zeta_m\zeta_n]$ goes to 0 as $m,n\to\infty$ with $\abs{m-n}\ge\sqrt n$. 

Let $\eps>0$. Almost surely, there exists $n_0$ such that, if $m,n\ge n_0$ and $\abs{n-m}\ge\sqrt n$, $\abs{E^\beta[\zeta_m\zeta_n]}<\eps$. For any $N>n_0$, we have, using $\abs{\zeta_m\zeta_n}\le 1$ for the first term, 
\begin{align*}
    &\abs[\bigg]{\frac1{N^2}\sum_{m=1}^N\sum_{n=1}^NE^\beta[\zeta_m\zeta_n]}\\
    & \le \frac1{N^2}\#\{(m,n)\st m,n\in\IeFer N, \text{and either }m,n\in\IeFer{n_0}\ \text{or}\ \abs{n-m}<\sqrt n\}
      +\frac1{N^2}\sum_{\substack{n_0\le m,n\le N\\\abs{n-m}>\sqrt n}}\eps\\
    & \le \frac1{N^2}\pa[\Big]{n_0^2+\sum_{n=n_0}^N2\sqrt n}+\eps = \frac1{N^2}O(N^{3/2})+\eps \le 2\eps,
\end{align*}
where the last inequality holds for large $N$. This shows that Condition $\text{(B}_\beta\text{)}$ holds for almost every~$\beta$, in the general case. 

\end{proof}

\section{I.I.\ representation, one-factor Markov case}\label{sec:II_1fMarkov}

In this section, we focus on a framework for which the family $(\mathfrak{b}_K)_K$ is no longer mutually independent. For each integer $k\in\N$, and with a slight abuse of notation, let $\mathfrak b_k = (\beta_{n,k})_{n\ge k\vee1}$ be a family of random variables, and consider the following assumption:

\begin{description}
    \item[(1fM)] The families $\mathfrak b_k$, $k\in\N$, are mutually independent, the random variables $\beta_{n,k}$ are Bernoulli random variables, and their parameters $p_{n,k}$ are elliptic, i.e.\ uniformly bounded away from 0 and 1: there exists $\epsilon\in(0,1/2)$ such that \[p_{n,k}\in[\epsilon,1-\epsilon]\] for all admissible $(n,k)$. And, for all $K\subset\IeFer n$, $\beta_{n,K}=\beta_{n,\abs K}$. 
\end{description}

This assumption is referred to as a \emph{one-factor Markov assumption}, in reference to the fact that the following two properties are equivalent:
\begin{enumerate}[(i)]
    \item for all $n$, there exists $f_n:\Z\to\Er$ such that, for all $x_1,\ldots,x_n\in\Er$,\ $\phi_n(x_1,\ldots,x_n)=f_n(x_1+\cdots+x_n)$,
\textit{i.e.}, a.s.,\ \[\eta_{n+1}=f_n(X_n)\xi_{n+1}.\]
    \item for all $n$, for all $K,L\subset\IeFer n$, $\abs{K}=\abs L\Rightarrow \beta_{n,K}=\beta_{n,L}$.
\end{enumerate}
As such, $((X_n,Y_n))_n$ is a Markov process, moreover $\eta_{n+1}=Y_{n+1}-Y_n$ depends on $(X_n,Y_n)$ through only one factor, $X_n$.

This equivalence is a direct consequence of the uniqueness of the representation and of the simple fact that any symmetric function of $(x_1,\ldots,x_n)\in\Er^n$ is a function of $(n,x_1+\cdots+x_n)$. 

Let us first emphasize in the following lemma that any attempt to reproduce the same strategy as previously is bound to fail in this case.
\begin{lemma}
    Assume (1fM) and furthermore that, for all $n\ge0$ and $0\le k\le n$, $p_{n,k}=p$ for some $p\in(0,1)$. 
    
    If $p\le\frac12$,  then  $\displaystyle\sum_n\E[E^\beta[\zeta_n]^2]=\infty$. 
\end{lemma}

\begin{proof}
Due to the assumption, 
\begin{align*}
\zeta_n
    & = (-1)^{\sum_{K\subset\{i\le n\st z_i=1\}}\beta_{n,K}}=(-1)^{\sum_{k=0}^{S_n}{S_n\choose k}\beta_{n,k}}=(-1)^{\sum_{k=0}^{n}{S_n\choose k}\beta_{n,k}}
\end{align*}
where $S_n=|\kappa_n(\xi)|=z_1+\cdots+z_n$ (and it is understood as usual that ${S_n\choose k}$ is zero as soon as $k>S_n$). Recall that $z_i=(1-\xi_i)/2$. Then, as before, 
\begin{align*}
    \E[E^\beta[\zeta_n]^2]
        & = \E[E^\beta[\zeta_n\zeta'_n]]
         = \E[\E[\zeta_n\zeta'_n\s \xi,\xi']]\\
         & = \E[\E[(-1)^{\sum_{k=0}^n ({S_n\choose k}+{S'_n\choose k})\beta_{n,k}}\s \xi,\xi']]\\
         & = \E[\E[(-1)^{\sum_{k=0}^n\ind{({S_n\choose k}+{S'_n\choose k}\text{ is odd})} \beta_{n,k}}\s \xi,\xi']]\\
         & = \E[(1-2p)^{\#\{k\le n\st {S_n\choose k}+{S'_n\choose k}\text{ is odd}\}}]
\end{align*}
Note in particular that, when $S_n=S'_n$, the sum ${S_n\choose k}+{S'_n\choose k}$ is obviously even, so that the term in the expectation is $(1-2p)^0=1$. 

Let us now assume that $p\le\frac12$ so that terms are nonnegative. We conclude that
\[
\E[E^\beta[\zeta_n]^2]
    \ge \P(S_n=S'_n).
\]
However, $\sum_n\P(S_n=S'_n)=\infty$, as can be seen in several ways. For instance, this is $\E[\#\{k\st S_k=S'_k\}]$, the expected number of intersections; however there are a.s.\ infinitely many intersections between two independent symmetric simple random walks on $\mathbb Z$. Or $\P(S_n=S'_n)\ge\frac1{n+1}$, as is the case for any distribution on a set of cardinality $n+1$ (by Cauchy-Schwarz inequality, $\sum_{k\le n} p_k^2 = \frac1{n+1}\sum_{k\le n}p_k^2\sum_{k\le n} 1^2\ge \frac1{n+1}(\sum_{k\le n} p_k)^2=\frac1{n+1}$). This concludes the proof. 
\end{proof}

However, a different strategy enables us to obtain annealed convergence results. Let us first consider the proof of Condition~(A), while Condition~(B) will be proved in later subsections under various assumptions. 

\subsection{Condition~(A)}

Let us consider both annealed and quenched settings for Condition (A). 

\begin{proposition}[I.I.\ representation]\label{prop:A_markov}\ 

a) Assume (1fM). Then  (A) holds. 

b) Assume (1fM) \textbf{and} that for all $k$, $(\beta_{n,k})_{n\ge k}$ are independent. Then $\Abeta$ holds for a.e.\ $\beta$. 
\end{proposition}

\begin{proof}
Let us first remark that proving b) implies a proof of a) even though the assumption for b) is a priori stronger. This is a consequence of the fact that Condition (A) depends only on the sequence $(\E[\zeta_n])_n$. Thus, the dependence structure of the variables $\beta_{n,k}$ across different values of $n$ is irrelevant to (A), and we may assume, without loss of generality, that they are independent across $n$. This independence across $n$, together with (1fM), is equivalent to the mutual independence of the random variables $\beta_{n,k}$, $n\in\N$, $k\in\N\cup\{0\}$.

In order to prove $\Abeta$, we shall use the fact that, by Kronecker's lemma \cite{Durrett}, for any $\beta$,
\[\Abeta\quad\Leftarrow\quad \sum_{n=1}^\infty\frac{E^\beta[\zeta_n]}n\text{ converges.}\]
For any $n\ge1$, the random variable $U_n:=E^\beta[\zeta_n]$ depends on $(\beta_{n,k})_{k\le n}$ only, hence $U_1$, $U_2$,... are independent. Furthermore, $\abs{U_n}\le1$ for all $n$. Therefore Kolmogorov's three-series theorem~\cite{Durrett} applies to $\sum_nU_n/n$. Both the ``tail'' and ``variance'' conditions of this theorem are in fact made trivial by the boundedness of $U_n$ so that it finally boils down to the ``expectation'' condition: 
\[\sum_{n=1}^\infty \frac{U_n}n\text{ converges a.s.}\quad\Leftrightarrow\quad \sum_{n=1}^\infty\frac{\mathbb E[U_n]}n\text{ converges.}\]
Since $\E[U_n]=\E[\zeta_n]$, we conclude that, to prove $\Abeta$, it suffices to prove the convergence of the series $\sum_n\frac1n{\E[\zeta_n]}$. 

Let us procede with the computation of $\E[\zeta_n]$. In the 1-factor Markov case, 
\begin{align*}
\zeta_n
    & = (-1)^{\sum_{K\subset\{i\le n\st z_i=1\}}\beta_{n,K}}=(-1)^{\sum_{k=0}^{S_n}{S_n\choose k}\beta_{n,k}}=(-1)^{\sum_{k=0}^{n}{S_n\choose k}\beta_{n,k}}
\end{align*}
where $S_n=z_1+\cdots+z_n$ (and it is understood as usual that ${S_n\choose k}$ is zero as soon as $k>S_n$). Then, as before, 
\begin{align}
\E[\zeta_n]
	& = \E[\E[\zeta_n\s \xi]]\notag\\
	& = \E\cro[\bigg]{\E[(-1)^{\sum_{k=0}^n\beta_{n,k}\ind{\{{S_n\choose k}\text{ is odd}\}}}\s \xi]}\notag\\
    & = \E\cro[\bigg]{\prod_{k\le n\st {S_n\choose k}\text{ is odd}}(1-2p_{n,k})}\notag
\end{align}
hence by triangle inequality and uniform bound on $p_{n,k}$,
\begin{align}
\abs[\big]{\E[\zeta_n]}
	& \le \E[(1-2\epsilon)^{\#\{k\le n\st {S_n\choose k}\text{ is odd}\}}].\label{eq:odd}
\end{align}
By Lucas theorem~\cite{Lucas}, for nonnegative integers $S$ and $k$, ${S\choose k}$ is odd if, and only if, in the binary representations of $k$ and $S$, the locations of the 1's for $k$ is a subset of the locations of the 1's for $S$. As a consequence, a.s., 
\begin{equation}\label{eq:bitcount}
\{k\le n\st {S_n\choose k}\text{ is odd}\} = 2^{N_n}
\end{equation}
where $N_n$ is the bit count of $S_n$, i.e.\ the number of ones in the binary representation of $S_n$. Although the distribution of the bits of $S_n$ varies widely with $n$ since $S_n$ concentrates around $n/2$, we shall argue that, because this concentration holds in a $\sqrt{n}$-window, the lower half of the bits of $S_n$ (with some margin) are almost uniform. This is made precise in Lemma~\ref{lem:coupling} in the appendix. 

More specifically, choose an arbitrary $\alpha\in(0,1)$. Denote by $\Nt_n$ the number of ones among the $\ell_n:=\floor{\frac{1-\alpha}2\log_2 n}$ lower bits of $S_n$, by $U_n$ the uniform r.v.\ obtained by Lemma~\ref{lem:coupling} (with this value of $\alpha$, and $S_0=0$) and by $K_n$ the bit count of $U_n$, so that $\Nt_n=K_n$ on $\{S_n\equiv U_n\modulo2^{\ell_n}\}$ and
\[\P(S_n \not\equiv U_n\modulo2^{\ell_n})\le 2e^{-\pi^2 n^\alpha}.\]

Since $U_n$ is uniformly distributed on $\{0,\ldots,2^{\ell_n}-1\}$, its last $\ell_n$ bits are independent Bernoulli r.v.'s of parameter $1/2$, so that $K_n$ is a binomial r.v.\ of parameters $\ell_n$ and $1/2$. 

One has, combining~\eqref{eq:odd}, \eqref{eq:bitcount} and Lemma~\ref{lem:coupling}, 
\begin{align}
	\abs[\big]{\E[\zeta_n]}  & \le \abs[\big]{\E[(1-2\epsilon)^{2^{N_n}}]}\notag\\
	& \le \E[\abs{1-2\epsilon}^{2^{N_n}}]\notag\\ 
	& \le \E[\abs{1-2\epsilon}^{2^{\Nt_n}}]\notag\\
	&  \le \P(S_n\not\equiv U_n\modulo 2^{\ell_n})+\E[\abs{1-2\epsilon}^{2^{K_n}}]\notag\\
	& \le 2e^{-\pi^2n^{\alpha}}+2^{-\ell_n}\sum_{k=0}^{\ell_n}{\ell_n\choose k}\abs{1-2\epsilon}^{2^k}\notag.\\
    & \le 2e^{-\pi^2n^\alpha}+n^{-\frac{1-\alpha}2}e^{O((\log\ell_n)^2)} = n^{-\frac{1-\alpha}2+o(1)}\label{eq:boundNn}
\end{align}
using Lemma~\ref{lem:subexponential} for the last inequality, 
since $(\log\ell_n)^2=O((\log\log n)^2)=o(\log n)$. As a consequence, 
the series $\sum_n\frac1n{\E[\zeta_n]}$ converges, proving $\Abeta$ for almost every $\beta$ by the initial discussion. 
\end{proof}

\subsection{$n$-independence case} Note that Proposition~\ref{prop:A_markov}, Part a), does not involve the dependence of the sequence $(\phi_n)_n$ as $n$ varies. Let us first assume independence, i.e.\ each new term $\phi_{n+1}$ involves a ``new sampling'' of the coefficients.  

\begin{proposition}[I.I.\ representation]\label{prop:condBMarkov}
Assume (1fM) \textbf{and} that for each $k$, $(\beta_{n,k})_{n\ge k}$ are independent. 
Then Condition (B) holds.
\end{proposition}

\begin{corollary}[I.I.\ representation]\label{cor:Markov} Assume (1fM) \textbf{and} that for each $k$, $(\beta_{n,k})_{n\ge k}$ are independent. Then convergence to a standard Brownian motion holds under the annealed measure $\P$: ``the random random walk is normal''. 
\end{corollary}

Remark. In the 1-factor Markov case, note that $\E[\zeta_n]$ was shown to decrease polynomially to 0, as opposed to the exponential convergence in the ``general dependence case''. This can be intuitively understood by the much more limited addition of randomness: only $n+1$ random parameters define the random walk $Y$ at step $n$, compared to $2^n$, making the prevalence of normality more subtle. This subtlety also echoes the fact that there are examples of environments $\beta$ where $(B_\beta)$ \textit{fails}, as was proved in~\cite{CHSW} (Lévy transformation). We only gave annealed results, and these remarks suggest that a quenched result (similar to Proposition~\ref{prop:B_quenched}) would be sensibly more difficult. 

\begin{proof}[Proof of Proposition~\ref{prop:condBMarkov}]
It turns out that (B) actually quickly follows from the computations performed to prove~(A). 

Let $N\in\N$. Let $m,n$ be integers such that $1\le m<n\le N$. 
Using the same reasoning as in~\eqref{eq:odd} and~\eqref{eq:bitcount}, together with the independence between $(\beta_{m,k})_{k\in\IeFer m}$ and $(\beta_{n,k})_{k\in\IeFer n}$, we have
\begin{align*}
\abs[\big]{\E[\zeta_m\zeta_n]}
	& = \abs[\Big]{\E\cro[\Big]{\E[\zeta_m\s \xi]\E[\zeta_n\s \xi]}}
	 \le\E[(1-2\epsilon)^{2^{N_m}+2^{N_n}}]
\end{align*}
hence
\begin{align*}
\abs[\big]{\E[\zeta_m\zeta_n]}
	 \le  \E[\abs{1-2\epsilon}^{2^{N_n}}]=:a_n,
\end{align*}
and $a_n$ is precisely the quantity that what bounded in the proof of~Proposition~\ref{prop:A_markov}. In particular, we deduce from that proof that $a_n\to0$ as $n\to\infty$. Thus, 
\begin{align*}
	\abs[\bigg]{\frac1{N^2}\sum_{1\le m<n\le N}\E[\zeta_m\zeta_n]}
	& \le \frac1{N^2}\sum_{1\le m<n\le N}a_n \le \frac1N\sum_{n=1}^N a_n\limites{}{N\to\infty}0
\end{align*}
by Cesàro's lemma, which implies (B). Note indeed that the diagonal terms ($m=n$) are all equal to~1 and thus account for a total of $N = o(N^2)$. 
\end{proof}

The previous proof suggests that Condition (B) is sometimes, paradoxically, almost a consequence of (A), in spite of the usual increase of both complexity and constraint when moving up from expectation to variance. Let us indeed state such a remarkable result, which especially owes its simple form to the restriction to weight parameters smaller than~$1/2$. In practice, as in the previous proof, not only Condition (A) but a slightly stronger result is typically obtained, which allows the result to carry over to larger parameters. We also state such a result in part b) below. 

Furthermore, the one-factor Markov model under investigation in this section is characterized by a grouping according to cardinality ($\beta_{n,K}=\beta_{n,L}$ whenever $|K|=|L|$). However, this setting is by no means a requirement to achieve convergence to a standard Brownian motion. In fact the grouping can be generalized to an arazbitrary partition of $2^{\IeFer n}$. More precisely, for each $n$, let
\[
\psi_n:2^{\IeFer n}\longrightarrow R_n
\]
be a surjective map onto a finite index set $R_n$. We assume that the weights are constant within the groups induced by $\psi_n$, so that
\begin{equation}\label{eq:psigrouping}
\psi_n(K)=\psi_n(L)
\quad\Longrightarrow\quad
\beta_{n,K}=\beta_{n,L},
\qquad K,L\subset\IeFer n.   
\end{equation}

$\beta_{n,k}$ will denote the common $\beta_{n,\cdot}$-value on the fibre
$\psi_n^{-1}(\{k\})$, $k\in R_n$. This general setting allows for groupings as fine as the discrete partition ($\psi_n$ is a bijective map -- this is the independent case) and as corse as the trivial partition ($\psi_n$ is constant). The 1-factor Markov case corresponds to the map $\psi_n:K\mapsto \abs K$.

Note that it may not be possible or natural to identify different $k$'s across different $n$'s and it may therefore not be possible to define a family $\mathfrak b_k$, $k\in\N\cup\{0\}$, as in the one-factor Markov case.

\begin{proposition}[I.I.\ representation]
Assume the sequence $(\beta_{n,K})_{n\in\N,\,K\subset\IeFer n}$ is independent in $n$ (i.e.\ the family $(\mathfrak{b}_n)_n$ is mutually independent, where $\mathfrak b_n=(\beta_{n,K})_{K\subset\IeFer n}$), and that the dependence in $K$ is such that for all $n$, the weights $(\beta_{n,K})_{K\subset\IeFer n}$ are constant within the groups induced by a map $\psi_n:2^{\IeFer n}\rightarrow R_n$. See \eqref{eq:psigrouping}.  
Furthermore, assume that the random values $\beta_{n,k}$, for $k\in R_n$, defined by $\beta_{n,k}=\beta_{n,K}$ for any $K\in\psi_n^{-1}(\{k\})$, are independent and Bernoulli distributed with respective parameters $p_{n,k}$, $k\in R_n$. 
\begin{enumerate}[a)]
    \item If, furthermore, $p_{n,k}\le 1/2$ for all $n,k$, then Condition (A) implies Condition (B).
    \item If Condition ($\widetilde A$): $\E[\zeta_n]\limites{}{n\to\infty}0$ (which is stronger that (A)) holds for any parameters $p_{n,k}\in[\epsilon,1/2]$ for some $\epsilon\in[0,1/2]$, then both conditions (A) and (B) hold for any parameters $p_{n,k}\in[\epsilon,1-\epsilon]$. 
\end{enumerate}
\end{proposition}

\begin{proof}
a) For all $n$, we have, due to the dependence structure, 
\begin{align*}
\zeta_n  = (-1)^{\sum_{K\in A_n(\xi)}\beta_{n,K}}
 & = (-1)^{\sum_{k\in R_n} \#\pa[]{ A_n(\xi)\cap\{\psi_n(\cdot)=k\}}\beta_{n,k}}\\
 & = (-1)^{\sum_{k\in R_n, \#( A_n(\xi)\cap\{\psi_n(\cdot)=k\})\text{ is odd}}\beta_{n,k}}
\end{align*}
where $A_{n}(\xi)=2^{\{i\in\IeFer n\st \xi_i=-1\}}$,
hence due to the law of the weights, and the assumption on their parameters,
\begin{align}
\E[\zeta_n\s\xi] & = \E\cro[\bigg]{\prod_{_{k\in R_n, \#( A_n(\xi)\cap\{\psi_n(\cdot)=k\})\text{ is odd}}}\underbrace{(1-2p_{n,k})}_{\ge0}\Big|\xi}\ge 0.\label{eq:Eprod}
\end{align}
On the other hand, obviously $\E[\zeta_n\s\xi]\le1$ since $\abs{\zeta_n}=1$ a.s.. 

Then, for any integers $m\ne n$, the independence assumption entails that $\zeta_m$ and $\zeta_n$ are independent given $\xi$, so that 
\begin{align}
\E[\zeta_m\zeta_n]
	& = \E\cro[\big]{\E[\zeta_m\s\xi]\E[\zeta_n\s\xi]},\label{eq:prodsxi}
\end{align}
and the previous bounds imply
\[
0 \le \E[\zeta_m\zeta_n] \le \E[\E[\zeta_n\s\xi]] = \E[\zeta_n].
\]
As a consequence, as $N\to\infty$, 
\begin{align*}
0\le \frac1{N^2}\sum_{\substack{m,n\in\IeFer N\\m\ne n}}\E[\zeta_m\zeta_n] \le \frac1{N^2}\sum_{\substack{m,n\in\IeFer N\\m\ne n}}\E[\zeta_n]\le\frac1N\sum_{n=1}^N\E[\zeta_n].
\end{align*}
Since as usual the diagonal terms are $\sum_{n=1}^N\E[\zeta_n^2]=N=o(N^2)$, we conclude that (A) implies (B). 

b) Let $(p_{n,k})_{n\in\N, k\in R_n}$ be parameters such that $p_{n,k}\in[\epsilon,1-\epsilon]$ for all $n,k$, where $\epsilon\in[0,\frac12]$. First, observe that, for all $n$, as in~\eqref{eq:Eprod}, for some random subset $B_n(\xi)\subset R_n$, 
\begin{align*}
    \abs[\Big]{\E[\zeta_n]}=\abs[\Big]{\E\cro[\big]{\E[\zeta_n\s\xi]}} 
    & = \abs[\bigg]{\E\cro[\bigg]{\prod_{k\in B_n(\xi)}(1-2p_{n,k})}} \\
     & \le \E\cro[\bigg]{\prod_{k\in B_n(\xi)}\abs{1-2p_{n,k}}} \\
    & = \E\cro[\bigg]{\prod_{k\in B_n(\xi)}(1-2q_{n,k})} = \E'[\E'[\zeta_n\s\xi]]=\E'[\zeta_n],
\end{align*}
where $q_{n,k}=1-p_{n,k}$ if $p_{n,k}>1/2$ and $q_{n,k}=p_{n,k}$ if $p_{n,k}\le1/2$, and $\E'$ refers to a probability measure $\P'$ on $\Omega$ which only differs from $\P$ by the fact that the variables $\beta_{n,k}$ have parameter $q_{n,k}$ instead of $p_{n,k}$. Note that $q_{n,k}\in[\epsilon,1/2]$ for all $n,k$, and that by the above bound, the following implication holds:
\[\frac1N\sum_{n=1}^N\E'[\zeta_n]\to0\qquad \Rightarrow\qquad \abs[\bigg]{\frac1N\sum_{n=1}^N \E[\zeta_n]}\le\frac1N\sum_{n=1}^N\abs[\Big]{\E[\zeta_n]}\to0.\]
Thus, ($\widetilde{\rm A}$) for all parameters in $[\epsilon,1/2]$ implies (A) for all parameters in $[\epsilon, 1-\epsilon]$. 

As for (B), the Cauchy-Schwarz inequality gives, for all $m\ne n$, as in~\eqref{eq:prodsxi},
\begin{align*}
\abs[\big]{\E[\zeta_m\zeta_n]}
	& = \abs[\big]{\E\cro[\big]{\E[\zeta_m\s\xi]\E[\zeta_n\s\xi]}} \le\E\cro[\big]{\E[\zeta_m\s\xi]^2}^{1/2}\E\cro[\big]{\E[\zeta_n\s\xi]^2}^{1/2},
\end{align*}
hence for all $N$, 
\begin{align}
\abs[\bigg]{\frac1{N^2}\sum_{\substack{m,n\in\IeFer N,\\m\ne n}}\E[\zeta_m\zeta_n]}
    & \le \frac1{N^2}\sum_{\substack{m,n\in\IeFer N,\\m\ne n}}\abs[\big]{\E[\zeta_m\zeta_n]} \notag\\
    & \le \pa[\bigg]{\frac1{N} \sum_{m\in\IeFer N}\E\cro[\big]{\E[\zeta_m\s\xi]^2}^{1/2}}^2, \label{eq:CS}
\end{align}
However, as in~\eqref{eq:Eprod}, for some random subset $B_n(\xi)\subset R_n$, 
\begin{align*}
    \E\cro[\Big]{\E[\zeta_n\s\xi]^2}
    & = \E\cro[\bigg]{\prod_{k\in B_n(\xi)}(1-2p_{n,k})^2} 
      = \E\cro[\bigg]{\prod_{k\in B_n(\xi)}(1-2\widetilde p_{n,k})} 
     = \E''[\E''[\zeta_n\s\xi]]=\E''[\zeta_n],
\end{align*}
where $\widetilde p_{n,k}=2p_{n,k}(1-p_{n,k})\in[\eps,1/2]$, and $\E''$ refers to a probability measure $\P''$ on $\Omega$ which only defers from $\P$ by the fact that the variables $\beta_{n,k}$ have parameter $\widetilde p_{n,k}$ instead of $p_{n,k}$. Thus, condition ($\widetilde A$) for the weights $\widetilde p_{n,k}$ gives $\E''[\zeta_n]\to0$ hence $\E[\E[\zeta_n\s\xi]^2]^{1/2}\to 0$ and Condition (B) follows from~\eqref{eq:CS} by Cesàro's lemma.
\end{proof}

\subsection{$n$-invariant case} Let us now consider a stronger dependence assumption, namely that $\beta_{n,K}$ does not vary with $n$. This can be thought of as a way of ``enlarging'' the transformation $\phi_n$ to construct $\phi_{n+1}$ instead of resampling every of its coefficients. 

\begin{proposition}[I.I.\ representation]\label{prop:condBMarkovdep}
Assume (1fM) \textbf{and} that $\beta_{n,k}=\beta_k$ for all $n\in\N$,  $0\le k\le n$ where $(\beta_k)_{k\in\N\cup\{0\}}$ is a family of independent Bernoulli random variables with respective parameters $(p_k)_{k\in\N\cup\{0\}}$ and that for each $k$, $p_k\in[\varepsilon,1-\varepsilon]$. Then Condition (B) holds. 
\end{proposition}

\begin{corollary}[I.I.\ representation]\label{cor:Markov-Depend} Assume (1fM) \textbf{and} that $\beta_{n,k}=\beta_k$ for all $n\in\N$, $0\le k\le n$,  where $(\beta_k)_{k\in\N\cup\{0\}}$ is a family of independent Bernoulli random variables with respective parameters $(p_k)_{k\in\N\cup\{0\}}$ and that for each $k$, $p_k\in[\varepsilon,1-\varepsilon]$. Then convergence to a standard Brownian motion holds under the annealed measure $\P$: ``the random random walk is normal''. 
\end{corollary}

\begin{proof}[Proof of Proposition~\ref{prop:condBMarkovdep}]
Let $N\in\N$. Let $0\le m<n\le N$. We have, as in Proposition~\ref{prop:A_markov}\footnote{Recall the convention that ${s\choose k}=0$ if $k>s$.}, 
\begin{align*}
\zeta_m\zeta_n
    & = (-1)^{\sum_{k=0}^{S_m}{S_m\choose k}\beta_{k}+\sum_{k=0}^{S_n}{S_n\choose k}\beta_k}
	=(-1)^{\sum_{k=0}^{m}{S_m\choose k}\beta_{k}+\sum_{k=0}^{n}{S_n\choose k}\beta_k} \\
	& = (-1)^{\sum_{k=0}^m\pa[\big]{{S_m\choose k}+{S_n\choose k}}\beta_k}(-1)^{\sum_{k=m+1}^n {S_n\choose k}\beta_k},\\
	& = (-1)^{\sum_{k\in\IeFer m\st  {S_m\choose k}+{S_n\choose k}\text{ odd}}\beta_k}(-1)^{\sum_{k\in\IeFer n\setminus\IeFer m\st {S_n\choose k}\text{ odd}}\beta_k}
\end{align*}

hence, using the distribution of $(\beta_0,\ldots,\beta_n)$, 
\begin{align*}
\E[\zeta_m\zeta_n] = \E[\E[\zeta_m\zeta_n\s \xi]]
	& = \E\cro[\Big]{\prod_{k\in\IeFer m\st  {S_m\choose k}+{S_n\choose k}\text{ odd}}(1-2p)\prod_{k\in\IeFer n\setminus\IeFer m\st {S_n\choose k}\text{ odd}\}}(1-2p)}. 
\end{align*}
We use the triangle inequality (or not, if $p\le1/2$) so terms become positive and we may derive a first bound:
\newcommand{\expo}{\text{Exponent}_{m,n}(\xi)}
\begin{align}
\abs[\big]{\E[\zeta_m\zeta_n]} 
	& \le \E\cro[\Big]{\abs{1-2\epsilon}^{\#\{k\in\IeFer m\st  {S_m\choose k}+{S_n\choose k}\text{ odd}\}}\abs{1-2\epsilon}^{\#\{k\in\IeFer n\setminus\IeFer m\st {S_n\choose k}\text{ odd}\}}}\notag\\
	& \le   \E\cro[\Big]{\abs{1-2\epsilon}^{\expo}}\notag\\
	& \le \P(S_n=S_m)+ \E\cro[\Big]{\abs{1-2\epsilon}^{\expo}, S_m<S_n}\label{eq:sumevenodd}
\end{align}
where we denote
\begin{align}
\expo 
	& :=
\#\{k\in\IeFer m\st  {S_m\choose k}\text{ even, }{S_n\choose k}\text{ odd}\}+\#\{k\in\IeFer n\setminus\IeFer m\st {S_n\choose k}\text{ odd}\}.\notag
\end{align}
First, classically, 
\[\P(S_m=S_n)=\P(S_{n-m}=0)=\frac1{2^{n-m}}{n-m\choose\frac{n-m}2}\ind{\{n\equiv m\modulo2\}} \le \frac C{\sqrt{n-m}},
\]
for some constant $C>0$. This implies, as $N\to\infty$, 
\begin{align}
\frac1{N^2}\sum_{0\le m<n<N}\P(S_m=S_n)
	\le \frac1{N^2}\sum_{0\le m<n<N}\frac{C}{\sqrt{n-m}} = \frac{C}{N^2}\sum_{k=1}^{N-1}\frac{N-k}{\sqrt{k}}\le \frac{C}N\sum_{k=1}^{N-1}\frac1{\sqrt k}\to0.\label{eq:limitequality}
\end{align}
\newcommand{\imn}{\mathfrak i_{m,n}}
Let us consider the last term in~\eqref{eq:sumevenodd}. On $\{S_m<S_n\}$, denote by $\imn$ the index of the highest (leftmost) bit that differs in the binary representations of $S_m$ to $S_n$, thus necessarily $S_n$ has a 1 and $S_m$ has a 0 at index $\imn$. This also implies that $S_n-S_m < 2^{\imn+1}$. 

Remember that, by Lucas' theorem \cite{Lucas}, the binomial coefficient ${l\choose k}$ is odd if, and only if, denoting by $k_r\cdots k_1k_0$ and $l_r\cdots l_1l_0$ the binary representations of $k$ and $l$ (padding with 0's on the left of the smaller number so that their lengths match), for $i=0,\ldots,r$ one has  $k_i=1$ $\Rightarrow$ $l_i=1$. 

Thus, for every $k$ that has a 1 at index $\imn$, the binomial coefficient ${S_m\choose k}$ is even. Let us denote by $K_{m,n}$ the set of such integers $k$, so that, on $\{S_m<S_n\}$, 
\[
\#\{k\in\IeFer m\st  {S_m\choose k}\text{ even, }{S_n\choose k}\text{ odd}\}
\ge 
\#\{k\in K_{m,n}\cap\IeFer m\st {S_n\choose k}\text{ odd}\}. 
\]

As a consequence, the exponent in~\eqref{eq:sumevenodd} can be bounded as follows, on $\{S_m<S_n\}$: 
\begin{align}
\expo
	& \ge \#\{k\in K_{m,n}\cap\IeFer m\st {S_n\choose k}\text{ odd}\}+\#\{k\in\IeFer n\setminus\IeFer m\st {S_n\choose k}\text{ odd}\}\notag\\
	& \ge \#\{k\in K_{m,n}\cap\IeFer n\st {S_n\choose k}\text{ odd}\} = 2^{N_n-1},\label{eq:Nnm1}
\end{align}

where $N_n$ is the bit count of $S_n$, i.e.\ the number of ones in the binary representation of $S_n$, so that $N_n-1$ is the number of bits of $S_n$ excluding $i_{m,n}$. Recall that both $S_n$ and $k\in K_{m,n}$ have a 1 at index $i_{m,n}$. Then \eqref{eq:Nnm1} comes from Lucas theorem: a number $k\in\IeFer n$ has ${S_n \choose k}$ odd if and only if the locations of its 1's (in its binary representation) is a subset of the locations of the 1's of $S_n$; since these locations characterize $k$, there are exactly $2^{N_n}$ values of $k$ for which ${S_n\choose k}$ is odd. If furthermore $k$ is restricted to $K_{m,n}$, then exactly 1 of the 1's of $S_n$ is ``forced'' in $k$, hence the correction by $-1$. 

The end of the proof then follows from the estimates of Proposition~\ref{prop:condBMarkov}. 
Choose any $\alpha\in(0,1)$. By~\eqref{eq:Nnm1} and the inequality in~\eqref{eq:boundNn}, we have
\begin{align*}
\E[\abs{1-2\epsilon}^{\expo},S_m<S_n]
	& \le \E[\abs{1-2\epsilon}^{2^{N_n-1}}]
	\le n^{-\frac{1-\alpha}2+o(1)}=: a_n\limites{}{n\to\infty}0.
\end{align*}
Then, as $N\to\infty$, getting back to the decomposition~\eqref{eq:sumevenodd} and using~\eqref{eq:limitequality} to get the ``$o(1)$'' below, we have
\begin{align*}
\abs[\bigg]{\frac1{N^2}\sum_{0\le m<n< N}\E[\zeta_m\zeta_n]}
	& \le \frac1{N^2}\sum_{0\le m<n< N}\abs[\big]{\E[\zeta_m\zeta_n]}\\
	& \le o(1)+\frac1{N}\sum_{n=1}^{N-1}a_n \limites{}{n\to\infty}0
\end{align*}
by Cesàro's lemma. 
\end{proof}

\section{Totally unordered representation}
\label{sec:totally_unordered}

Let us consider the T.U.\ representation, i.e. 
\[\zeta_n = (-1)^{\beta_{n,\kappa_n(\xi)}}\]
where as usual $\kappa_n(\xi)=\{i\in\IeFer n\st \xi_i=-1\}$. 

\subsection{$n$-independent case}

Assume momentarily that $(\beta_{n,K})_{n,K}$ are i.i.d., or at least independent as $n$ varies and sharing same marginal distribution. In this case, the \emph{annealed} law of $(X,Y)$ is simple. Indeed, under $\P=\E[\P^\beta(\cdot)]$, for all $n$, 
\[
(\xi_{n+1},\eta_{n+1})=(\xi_{n+1}, (-1)^{\beta_{n,\kappa_n(\xi)}}\xi_{n+1})
\]
is independent of $(X_1,Y_1),\ldots,(X_n,Y_n)$ (this is Lemma~\ref{lem:tot_unordered} below) and distributed as two Rademacher r.v.'s with parameter $\frac12$ and covariance
\[\E[\xi_{n+1}\eta_{n+1}]=\E[(-1)^{\beta_{n,\kappa_n(\xi)}}]=1-2p.\]
Then, Donsker's theorem directly shows that $W^{(n)}$ converges to a Brownian motion with marginal variance 1 and covariance $1-2p$. 

Let us give a slightly finer statement.

Let, for $n\in\N$,  $\mathfrak b_n=(\beta_{n,K})_{K\subset\IeFer n}$, and assume that

\begin{description}
    \item[(H)] the family $(\mathfrak b_n)_{n\in\N}$ is mutually independent and for any given $n$, all $\beta_{n,K}$, $K\subset\IeFer n$, have same Bernoulli distribution with parameter $p_n$, $\Br(p_n)$.
\end{description}

\begin{proposition}[T.U.\ representation]~\label{lem:tot_unordered}
Under Assumption (H), the sequence $(\zeta_n)_n$ is independent under $\P$. 
\end{proposition}

\begin{proof}
For any $n$, and $z_1,\ldots,z_n\in\sE$, we have
\begin{align*}
 & \hspace{-1cm}  \P(\zeta_1=z_1,\ldots,\zeta_n=z_n)  \\
	& = \P((-1)^{\beta_{1,\kappa_1(\xi)}}=z_1,\ldots,(-1)^{\beta_{n,\kappa_n(\xi)}}=z_n)\\
	& = \sum_{L_1,\ldots,L_n} \P((-1)^{\beta_{1,L_1}}=z_1,\ldots,(-1)^{\beta_{n,L_n}}=z_n, \kappa_1(\xi)=L_1,\ldots,\kappa_n(\xi)=L_n)\\
	& = \sum_{L_1,\ldots,L_n} \P((-1)^{\beta_{1,L_1}}=z_1,\ldots,(-1)^{\beta_{n,L_n}}=z_n)\P(\kappa_1(\xi)=L_1,\ldots,\kappa_n(\xi)=L_n)\\
	& = \sum_{L_1,\ldots,L_n} \P((-1)^{\beta_{1,L_1}}=z_1)\cdots\P((-1)^{\beta_{n,L_n}}=z_n)\P(\kappa_1(\xi)=L_1,\ldots,\kappa_n(\xi)=L_n)\\
	& = \sum_{L_1,\ldots,L_n} \P((-1)^{\beta_{1,\emptyset}}=z_1)\cdots\P((-1)^{\beta_{n,\emptyset}}=z_n)\P(\kappa_1(\xi)=L_1,\ldots,\kappa_n(\xi)=L_n)\\
	& = \P((-1)^{\beta_{1,\emptyset}}=z_1)\cdots\P((-1)^{\beta_{n,\emptyset}}=z_n)
\end{align*}
and this product form classically implies the claim.
\end{proof}

\begin{proposition}[T.U.\ representation] \label{tot_unord_annealed}
Assume Assumption (H). Then, under the annealed law, 
\[\text{$W^{(n)}$ converges to a Brownian motion as $n\to\infty$}\quad\Leftrightarrow\quad\frac1N\sum_{n=1}^N p_n\limites{}{N\to\infty}p\in[0,1],\]
and under this condition the limit of $W^{(n)}$ is a Brownian motion with marginal variance 1 and covariance $1-2p$. 
\end{proposition}

Note that the assumptions are almost ``orthogonal'' to those of the General Case (G) for the previous representation: independence is assumed here as $n$ varies, not $K$; and the parameters $p_{n,K}$ satisfy much stricter conditions here, and play a role in the limit. 

Note also that at $p=1/2$ and under the assumption of independence as $n$ varies (as above), both results deal with the same model since for all $n$ the function $\phi_n$ is uniformly distributed, and $\phi_1,\phi_2,\ldots$ form an independent sequence.

\begin{proof}
The direct implication follows straight from Condition (A) of Theorem~\ref{mainth}, since $\E[\zeta_n]=1-2p_n$:
\[
(A)\quad\Rightarrow\quad \frac1N\sum_{n=1}^N(1-2p_n)\limites{}{N\to\infty}\rho\in\R
\quad\Rightarrow\quad \frac1N\sum_{n=1}^N p_n\limites{}{N\to\infty} \frac{1-\rho}2=:p\in[0,1].
\]

Let us now assume $\frac1N\sum_{n=1}^Np_n\to p$. 
As mentioned above, the proof then relies on the independence, under $\P$, of the sequence $(\zeta_n)_n$.

This is shown in Lemma \ref{lem:tot_unordered} and allows to apply the three-series theorem to prove that the following series of independent terms converges almost surely:
\[\sum_{n=1}^\infty\frac{\zeta_n-\E[\zeta_n]}n.\]
By Kronecker's lemma, it follows that
\[\text{$\P$-a.s.,}\qquad\frac1N\sum_{n=1}^N(\zeta_n-\E[\zeta_n])\limites{}{N\to\infty}0.\]
Due to the assumption and $\E[\zeta_n]=1-p_n$, this entails
\begin{align}
\text{$\P$-a.s.,}\qquad\frac1N\sum_{n=1}^N\zeta_n\limites{}{N\to\infty}1-2p,\label{eq:cvps}
\end{align}
which implies the convergence in law (see Criterion (C)). 
\end{proof}

The annealed case turns out to also imply that the quenched distribution behaves similarly.

\begin{proposition}[T.U.\ representation]\label{tot_unord_quenched}
Assume Assumption (H) and $\frac1N\sum_{n=1}^Np_n\limites{}{N\to\infty}p\in[0,1]$.
Then, for $\P$-almost every $\beta$, $W^{(n)}$ converges to a Brownian motion with marginal variance 1 and covariance $1-2p$. 
\end{proposition}

\begin{proof}
In the course of the proof of Proposition~\ref{tot_unord_annealed}, we proved the almost sure convergence~\eqref{eq:cvps}, which can be rephrased as follows:
\[\E\cro[\bigg]{P^\beta\pa[\bigg]{\frac1N\sum_{n=1}^N\zeta_n\limites{}{N\to\infty}1-2p}}=1.\]
This however entails that the inner probability equals 1 $\P$-a.s., i.e.
\[\text{for $\P$-a.e.\ $\beta$,\ \  $P^\beta$-a.s.,}\qquad\frac1N\sum_{n=1}^N\zeta_n \limites{}{N\to\infty}1-2p.\]
This proves condition $\Cbeta$ for $\P$-a.e.~$\beta$, which concludes. 
\end{proof}

As a particular case, notice that the above applies if we further assume $\beta_{n,K}=\beta_n$ for all $K$, and $p_{n}=p$ to simplify the statement. 

\begin{corollary}[T.U.\ representation]\label{cor:tot_unord_quenched_equal}
Assume $(\beta_n)_n$ are independent Bernoulli($p$) random variables and, for all $n\in\N$, $K\subset\IeFer n$, $\beta_{n,K}=\beta_n$. Then, $\P$-almost every $\beta$, $W^{(n)}$ converges to a Brownian motion with marginal variance 1 and covariance $1-2p$.
\end{corollary}

\subsection{One-factor Markov case}

Recall (see Section~\ref{sec:II_1fMarkov}) that the one-factor Markov case (i.e.\ for all $n\ge1$, $\zeta_n=f_n(X_n)$ for some function $f_n$) amounts to $\beta_{n,K}=\beta_{n,L}$ as soon as $\abs K=\abs L$ (both in I.I.\ and T.U.\ cases). 

Remarkably, the previous results, both annealed and quenched, already encompass this case, since any correlation among the variables $\beta_{n,K}$ where $K\subset\IeFer n$ is allowed, for all $n$. 

\subsection{$n$-constant case}

In order to investigate the effect of stronger correlation in the bootstrap mechanism as $n$ varies, suppose now $\beta_{n,K}$ does not depend on $n$, i.e.\ there is a family $(\beta_K)_{K\subset\N,\ \text{finite}}$ of i.i.d.\ random variables with distribution $\Br(p)$, and set 
\[\beta_{n,K}=\beta_K\quad\text{for $n\ge1$, $K\subset\IeFer n$.}\]
Due to the particular form of the T.U.\ representation, we shall see that this case still remains simple. 

\begin{proposition}[T.U.\ representation]\label{tot_unorder_Kindep}
    Let $(\beta_K)_{K\subset\N,\ \text{finite}}$ be a family of i.i.d.\ Bernoulli random variables with parameter $p$. Assume $\beta_{n,K}=\beta_K$ for all $n\in\N$, $K\subset\IeFer n$. 
 
    Then both annealed and a.s.\ quenched convergence hold for $W^{(n)}$ toward a Brownian motion with correlation $1-2p$. 
\end{proposition}

\begin{proof}
One indeed has, under $\P$, 
\begin{align*}
\text{for all $n\ge1$,}\qquad
\zeta_{n+1} & = (-1)^{\beta_{\kappa_{n+1}(\xi)}} = 
\begin{cases}
    (-1)^{\beta_{\kappa_n(\xi)}}=\zeta_n, & \text{if $\xi_{n+1}=+1$}\\
    (-1)^{\beta_{\kappa_{n}(\xi)\cup\{n+1\}}}\text{ independent of $\zeta_1,\ldots,\zeta_n$,} & \text{if $\xi_{n+1}=-1$}. 
\end{cases}
\end{align*}
Thus, under $\P$, $(\zeta_n)_n$ is a Markov chain on $\sE$ where at each step $n$, $\zeta_{n+1}$ is either equal to $\zeta_n$ or to an independent copy of $\zeta_n$, with equal probabilities. This Markov chain is ergodic, with the distribution of $(-1)^{\beta_K}$ (for any $K$) as invariant distribution. By the ergodic theorem, 
\[
\text{$\P$-a.s.,}\qquad\frac1N\sum_{n=1}^N\zeta_n\limites{}{N\to\infty}\E[(-1)^{\beta_\emptyset}]=1-2p. 
\]
This implies Condition (C). And since this convergence is $\P$-almost sure, it is also $P^\beta$-almost sure for a.e.\ $\beta$, so that $\Cbeta$ holds for almost every $\beta$. The annealed and quenched convergences follow. 
\end{proof}

\textit{Remark.} Instead of assuming that for each $K$, the sequence $(\beta_{n,K})_{n\ge\max K}$ is a.s.\ constant (equal to $\beta_K$), the same result could be obtained under the weaker assumption that for every $K$, the sequence $(\beta_{n,K})_{n\ge\max K}$ is \emph{stationary} (i.e.\ its distribution is shift invariant) with $\Br(p)$ marginal, and that these sequences are independent of each other. Although $(\zeta_n)_n$ is not a Markov chain anymore, it still has renewal times at instants $T_1,T_2,\ldots$ such that $\xi_{T_k}=-1$, so that splitting $\zeta_1+\cdots+\zeta_n$ at these renewal times leads to a sum of i.i.d.~terms whose expectation is $2(1-2p)$ by Wald's identity, which enables to conclude, via a law of large numbers for independent, non identically distributed r.v.'s (e.g.~via Kolmogorov's two-series theorem). 

\newpage

\appendix

\section{Representations of boolean functions}\label{app:boolean_functions}

In this appendix, we argue the choices of the two bases systems, inclusion-induced (I.I.), and totally unordered (T.U.).
The results of this appendix are not used in the main sections of the paper apart from the definitions of the representations, which are already mentioned in the introduction. It can therefore be skipped on first reading.

\subsection{Choosing a basis system}\label{sec:consistency}
As expressed in the introduction, our aim is to develop a more flexible framework that accommodates richer forms of dependence than the one given by a sequence $(\phi_n)_n$ that is independent and uniformly distributed: $\mathbb{P}(\phi_n=f_n) = 2^{-{2^n}}$, for any $f_n\in\sE^{\sE^n}$.

To this end, we recall that the space $\mathbb{B}_n$ of Boolean functions $\phi:\sE^n \to \sE$ can be endowed with the structure of a vector space over the field $\mathbb{F}_2 = \mathbb{Z}/2\mathbb{Z}$. In this representation, the zero vector corresponds to the constant function equal to $1$, and the algebraic operations are defined by
\[
\phi \oplus \psi = \phi \psi, 
\qquad
\beta \odot \phi =
\begin{cases}
1 & \text{if } \beta = 0, \\
\phi & \text{if } \beta = 1.
\end{cases}
\]
This space has dimension $2^n$. Our model relies on the choice of a sequence $\pa[\big]{(g_a)_{a\in\Er^n}}_{n\in\N}$ of bases of $\mathbb B_1$, $\mathbb B_2$,... respectively. 

To incorporate the time dimension into the choice of bases, we impose additional structure. 

For $x \in \sE^n$, define
\[
x_+ = (x,1) \in \sE^{n+1}, 
\qquad 
x_- = (x,-1) \in \sE^{n+1}.
\]
A family of bases
\[
\bigl\{(g_a)_{a \in \sE^n},\ n \in \mathbb{N}\bigr\}
\]
is said to be \emph{consistent} if $g_{a_\pm}(x_\pm)$ is either 1 or $g_a(x)$, uniformly in $a,x\in \sE^n$; that is, there exist $q_-,q_+,r_-,r_+\in\{0,1\}$ such that, for every $n \in \mathbb{N}$ and all $a\in \sE^n$, $g_{a_+}$ and $g_{a_-}$ restricted to $\sE^n\times\{1\}$ and $\sE^n\times\{-1\}$, and seen as functions on $\sE^n$, take the form

\begin{center}
\begin{tabular}{||l||c|c||}\cline{2-3}
\multicolumn{1}{c||}{} & $\sE^n\times\{-1\}$ & $\sE^n\times\{1\}$\\\hline
$g_{a_-}$ & $g_a^{q_-}$ & $g_a^{r_-}$\\\hline
$g_{a_+}$ & $g_a^{r_+}$ & $g_a^{q_+}$\\\hline
\end{tabular}
\end{center}

Starting at $n=0$ and setting $g_{\varpi}(\varpi)=-1$, where $\varpi$ is the empty tuple ($\varpi\in \sE^0$), a choice of $(q_-,q_+,r_-,r_+)\in\{0,1\}^4$ defines a consistent basis system if and only if for every $n \in \mathbb{N}$, the family $(g_{a_-},g_{a_+})_{a \in \sE^n}$ is linearly independent, that is
\begin{equation}\label{eq:consistent}
\prod_{a\in \sE^n}\big(g_{a_-}^{\alpha_{a_-}}g_{a_+}^{\alpha_{a_+}}\big)=1 \Leftrightarrow\forall a\in \sE^n,\ \alpha_{a_-}=\alpha_{a_+}=0.
\end{equation}
Applied to $x_+$ and $x_-$, for all $x$, and using the assumption that $(g_a)_{a\in \sE^n}$ is a basis, the above gives
$$\forall a\in \sE^n,\quad\left\{\begin{array}{lcl}
q_-\alpha_{a_-}+r_+\alpha_{a_+}=0\\
r_-\alpha_{a_-}+q_+\alpha_{a_+}=0\\
\end{array}
\right.
\Leftrightarrow\forall a\in \sE^n,\ \alpha_{a_-}=\alpha_{a_+}=0.$$
This is equivalent to requesting that $q_-q_+\neq r_-r_+$.

Therefore only 6 possible consistent basis systems exist.
\begin{enumerate}

\item \label{case:totallyunordered}\underline{$(q_-,q_+,r_-,r_+)=(1,1,0,0)$}
\begin{center}
\begin{tabular}{||l||c|c||}\cline{2-3}
\multicolumn{1}{c||}{} & $\sE^n\times\{-1\}$ & $\sE^n\times\{1\}$\\\hline
$g_{a_-}$ & $g_a$ & 1\\\hline
$g_{a_+}$ & 1 & $g_a$\\\hline
\end{tabular}
\quad that is\quad
$\displaystyle
g_a(x) =
\begin{cases}
-1 & \text{if } x = a, \\
1 & \text{otherwise}.
\end{cases}
$
\end{center}
This corresponds to the totally unordered case.

\item \underline{$(q_-,q_+,r_-,r_+)=(0,0,1,1)$}
\begin{center}
\begin{tabular}{||l||c|c||}\cline{2-3}
\multicolumn{1}{c||}{} & $\sE^n\times\{-1\}$ & $\sE^n\times\{1\}$\\\hline
$g_{a_-}$ & 1 & $g_a$\\\hline
$g_{a_+}$ & $g_a$ & 1\\\hline
\end{tabular}
\quad that is\quad
$\displaystyle
g_a(x) =
\begin{cases}
-1 & \text{if } x = -a, \\
1 & \text{otherwise}.
\end{cases}
$
\end{center}
This corresponds to the relabeling ($a\to-a$) of the totally unordered basis system.

\item \label{case:inclusioninduced}\underline{$(q_-,q_+,r_-,r_+)=(1,1,0,1)$}
\begin{center}
\begin{tabular}{||l||c|c||}\cline{2-3}
\multicolumn{1}{c||}{} & $\sE^n\times\{-1\}$ & $\sE^n\times\{1\}$\\\hline
$g_{a_-}$ & $g_a$ & 1\\\hline
$g_{a_+}$ & $g_a$ & $g_a$\\\hline
\end{tabular}
\quad that is\quad
$\displaystyle
g_a(x)=
\begin{cases}
-1 & \text{if } \kappa(x)\supset \kappa(a), \\
1  & \text{otherwise}
\end{cases}
=\max_{k\in\kappa(a)}x_k
$
\end{center}
where $\kappa(x)=\{k\in\IeFer{n}:x_k=-1\}$.
This corresponds to the inclusion-induced order.

\item \underline{$(q_-,q_+,r_-,r_+)=(1,0,1,1)$}
\begin{center}
\begin{tabular}{||l||c|c||}\cline{2-3}
\multicolumn{1}{c||}{} & $\sE^n\times\{-1\}$ & $\sE^n\times\{1\}$\\\hline
$g_{a_-}$ & $g_a$ & $g_a$\\\hline
$g_{a_+}$ & $g_a$ & 1\\\hline
\end{tabular}
\quad that is\quad
$\displaystyle
g_a(x) =
\begin{cases}
-1 & \text{if } \kappa(x)\supset\kappa(-a), \\
1 & \text{otherwise}.
\end{cases}
=\max_{k\in\kappa(-a)}x_k
$
\end{center}
This case corresponds to the relabeling ($a\to-a$) of the inclusion-induced basis system.

\item \label{case:reverserinclusioninduced}\underline{$(q_-,q_+,r_-,r_+)=(1,1,1,0)$}
\begin{center}
\begin{tabular}{||l||c|c||}\cline{2-3}
\multicolumn{1}{c||}{} & $\sE^n\times\{-1\}$ & $\sE^n\times\{1\}$\\\hline
$g_{a_-}$ & $g_a$ & $g_a$\\\hline
$g_{a_+}$ & 1 & $g_a$\\\hline
\end{tabular}
\quad that is\quad
$\displaystyle
g_a(x)=
\begin{cases}
-1 & \text{if } \kappa(x)\subset \kappa(a), \\
1  & \text{otherwise}
\end{cases}
=-\min_{k\in\kappa(-a)}x_k
$
\end{center}
This corresponds to the reverse inclusion-induced order.

\item \underline{$(q_-,q_+,r_-,r_+)=(0,1,1,1)$}
\begin{center}
\begin{tabular}{||l||c|c||}\cline{2-3}
\multicolumn{1}{c||}{} & $\sE^n\times\{-1\}$ & $\sE^n\times\{1\}$\\\hline
$g_{a_-}$ & 1 & $g_a$\\\hline
$g_{a_+}$ & $g_a$ & $g_a$\\\hline
\end{tabular}
\quad that is\quad
$\displaystyle
g_a(x) =
\begin{cases}
-1 & \text{if } \kappa(x)\subset\kappa(-a), \\
1 & \text{otherwise}.
\end{cases}
=-\min_{k\in\kappa(a)}x_k
$
\end{center}
This corresponds to the relabeling ($a\to-a$) of the reverse inclusion-induced order.

\end{enumerate}
Note that $\kappa(-a)=\kappa(a)^c$ so that the condition $\kappa(x)\subset\kappa(-a)$ translates to $\kappa(x)\cap\kappa(a)=\emptyset$, and $\kappa(-x)\subset\kappa(a)$ to $\kappa(x)\cup\kappa(a)=\IeFer{n}$. Also, $x=a$ if and only $\kappa(x)=\kappa(a)$ so that the 6 basis systems are based on a comparison of $\kappa(x)$ to $\kappa(\pm a)$ using the relations $=$, $\supset$ or $\subset$. 

One may also stipulate that for every $n\in\mathbb{N}$ and all $a,x\in\sE^n$, $g_{a_+}(x_+) = g_{a_-}(x_-) = g_a(x)$, which expresses a natural compatibility of the basis system with the recursive construction of the hypercube. In this case one must have $q_-=q_+=1$, and only three basis systems remain: the totally unordered case (\ref{case:totallyunordered}), the inclusion-induced order (\ref{case:inclusioninduced}) and the reverse inclusion-induced order (\ref{case:reverserinclusioninduced}).

Among these, the totally unordered and inclusion-induced systems play a primal role. As explained in the next subsection, they coincide with two classical bases of the vector space of Boolean functions: the the canonical basis and the basis of monomials underlying the Algebraic Normal Form. This connection provides an additional justification for focusing on
these two representations throughout the paper.

Note that it will be natural to transfer the indexing from $\mathcal{E}^n$ to $2^{\IeFer{n}}$. Accordingly, the family $(g_a)_{a\in\mathcal{E}^n}$ is reindexed as $(g_K)_{K\subseteq\IeFer{n}}$, via the previous bijection $\kappa:\Er^n\to2^{\IeFer n}$. 

\subsection{Connection to Algebraic Normal Form}
\label{sec:representations}

The purpose of this subsection is to explain why the totally unordered and inclusion-induced basis systems identified above are particularly natural. Beyond arising from the consistency requirements of the previous subsection, they coincide with two fundamental representations of Boolean functions. The totally unordered basis corresponds to the canonical basis of $\F_2^{\sE^n}$, while the inclusion-induced basis corresponds to the monomial basis underlying the Algebraic Normal Form (ANF). We now recall these two representations and their algebraic interpretation.

Note first that, for any $n$, for a given family $(g_K)_{K\subset\N}$ of functions $\Er^n\to\Er$, one may equivalently write, for any \{0,1\}-valued coefficients $(\beta_{n,K})_{K\subset\IeFer n}$,
\[
\phi_n = \prod_{K\subset\IeFer n}g_K^{\beta_{n,K}}\qquad\text{or}\qquad
\phit_n=\sum_{K\subset\IeFer n}\beta_{n,K}\gt_K,\] 
where $\widetilde{-1}=1$ and $\widetilde{+1}=0$, i.e.~for $a\in\sE$, $\widetilde a\in\F_2$ is such that $(-1)^{\widetilde a}=a$. 
In other words, existence and uniqueness of such a representation of a function $\phi_n:\Er^n\to\Er$ amounts to the fact that $(\gt_K)_{K\subset\IeFer n}$ is a basis of the $\F_2$-vector space $\F_2^{\sE^n}$. 

$\bullet$ The most obvious basis is the \textbf{canonical basis} $(\et_K)_{K\subset\IeFer n}$ of $\F_2^{\sE^n}$, i.e.\ indicator functions of singletons: for all $K\subset\IeFer n$,
\[\forall L\subset\IeFer n,\qquad \et_K(u_L)=\ind{(K=L)}.\]
The decomposition of any $\phi_n:\sE^n\to\sE$ in this basis reads as follows: 
\[
\forall x=(x_1,\ldots,x_n)\in\sE^n,\qquad \phi_n(x)=\prod_{K\subset\IeFer n}e_K(x)^{\beta_{n,K}}=(-1)^{\beta_{n,\kappa(x)}}.
\]
This decomposition was referred to as the \emph{totally unordered} case in~\cite{CHSW}, where it was little used. While legitimate, this decomposition turns out to behave singularly when we randomly sample $\phi_n$ by sampling i.i.d.\ $(\beta_{n,K})_{K\subset\IeFer n}$, in that the joint sequence $\pa[\big]{(\xi_n,\eta_n)}_{n\ge1}$ is independent, which greatly simplifies the analysis. This is the topic of Section~\ref{sec:totally_unordered}. 

\begin{figure}
\[
  \begin{tikzcd}
    x\in\sE^n \arrow{r}{\phi_n} \arrow{dr}{\phit_n} \arrow{d}{\sim} & \sE \arrow{d}{\sim} \\
   \xt\in(\F_2)^n\arrow{r}{\phih_n}  & \F_2
  \end{tikzcd}
\]
\caption{Summary of notations in Section~\ref{sec:representations}: depending on the identification between $a\in\sE$ and $\widetilde a\in\F_2$ via $a=(-1)^{\widetilde a}$, every function $\phi_n:\sE^n\to\sE$ may equivalently be viewed as being defined on $(\F_2)^n$ and/or taking values in $\F_2$, and is then denoted by $\phih_n$ or $\phit_n$ as sketched.  }
\label{fig:cd}
\end{figure}

$\bullet$ Since $\sE^n$ can be identified with $(\F_2)^n$, every function $\phit_n:\sE^n\to\F_2$ can be viewed as a function $\phih:(\F_2)^n\to\F_2$ (i.e. $\phih(\xt)=\phit(x)=\widetilde{\phi(x)}$, see also Figure~\ref{fig:cd}). Because $(\F_2)^n$ is finite, every such function coincides with a polynomial in $n$ variables (as can be seen for instance using interpolation polynomials). Noting that, for $\eps\in\F_2$, $\eps^n=\eps$ for every $n\ge1$, we deduce that every polynomial function, hence every function $\phih:(\F_2)^n\to\F_2$, decomposes over monomials $\eps\mapsto \eps_{i_1}\cdots \eps_{i_k}$, i.e.\ as follows:
\[\forall \eps_1,\ldots,\eps_n\in\F_2,\qquad \phih(\eps_1,\ldots,\eps_n)=\sum_{K\subset\IeFer n}\beta_{n,K}\prod_{i\in K}\eps_i\]
for some coefficients $\beta_{n,K}\in\F_2=\{0,1\}$,  $K\subset\IeFer n$. 
The existence of this decomposition and the size $2^n$ of the family of monomials $\Ph_K:\eps\mapsto \prod_{i\in K}\eps_i$, where $K\subset\IeFer n$, which matches the dimension of $(\F_2)^{\sE^n}$, proves that this family is a basis, known as the \textbf{basis of monomials}, and hence that this decomposition is unique. 

This decomposition is ubiquitous in the theory of boolean functions (see, for instance, \cite[Theorem 1.33]{booleanfunctions} and references therein) where depending on context it is known as Algebraic Normal Form, Ring Sum Expansion, Zhegalkin polynomial, or Reed-Muller expansion. 

Translated for functions $\phi_n:\sE^n\to\sE$, this decomposition reads:
\[\forall x_1,\ldots,x_n\in\sE,\qquad \phi_n(x_1,\ldots,x_n)=\prod_{K\subset\IeFer n}(\max_{i\in K}x_i)^{\beta_{n,K}}.\]
Note indeed that $\prod_{i\in K} \xt_i = 1$ if and only if $\xt_i=1$ for all $i\in K$, i.e.\ $x_i=-1$ for all $i\in K$, in other terms $\max_{i\in K}x_i = -1$. This decomposition is thus no other than the main decomposition considered in~\cite{CHSW}, which there corresponded to the inclusion order among subsets. 

\textit{Remark.} As a side note, it turns out that the basis of monomials is closely related to the canonical basis in that they are Möbius dual of each other:
\[\forall K,L\subset\IeFer n,\qquad \Ph_K(\xt_L)=\sum_{M\subset L}\eh_K(\xt_M)\qquad\text{and}\qquad\eh_K(\xt_L)=\sum_{M\subset L}\Ph_K(\xt_M),\]
which entails for instance that one may obtain the coefficients $(\beta_{n,K})_{K\subset\IeFer n}$ of $\phih_n$ on the basis of monomials as
\[\forall K\subset\IeFer n,\qquad\beta_{n,K}=\sum_{M\subset K}\phih_n(\xt_M),\]
i.e.\ 
\begin{equation}\label{eq:inversion}
\forall K\subset\IeFer n,\qquad(-1)^{\beta_{n,K}}=\prod_{M\subset K}\phi_n(x_M).
\end{equation}
The family of representations introduced in~\cite[page 4]{CHSW} can actually fit within this scheme: for any partial order $\preceq$ on $\sub n$, the family $(g_K)_{K\subset\IeFer n}$ of functions $\sE^{\sE^n}\to\sE$ defined by
\[
\forall K,L\subset\IeFer n,\qquad \widehat g_K(\xt_L)=\sum_{M\preceq L} \widehat e_K(\xt_M) = \ind{(K\preceq L)}
\]
constitutes a basis of $(\F_2)^{\sE^n}$. The canonical basis $(\widehat e_K)_K$ corresponds to the case of a totally unordered set (where $K\subset L$ if and only if $K=L$), while the monomial basis $(\widehat P_K)_K$ corresponds to the usual inclusion order. However in this setting the duality is in general not symmetric (the inversion formula involves the Möbius-Rota function of the partial order). 

In this paper, we investigate bootstrap schemes obtained by randomizing the coordinates $\beta_{n,K}$
appearing in these representations. Although the totally unordered and
inclusion-induced representations parameterize the same space of Boolean
functions, they encode randomness in fundamentally different ways.
Indeed, in the totally unordered representation, $\displaystyle\phi_n(x)=(-1)^{\beta_{n,\kappa(x)}}$,
so that the value of $\phi_n(x)$ is determined by a single coefficient. By contrast, in the inclusion-induced representation, $\displaystyle\phi_n(x)=(-1)^{\sum_{K\subset\kappa(x)}\beta_{n,K}}$, and therefore depends on the aggregate contribution of all coefficients indexed by subsets of $\kappa(x)$. Each representation thus induces a different mechanism of randomization of the underlying Boolean function.

The canonical basis and the basis of monomials, although far from being the only possible bases of $\F_2^{\sE^n}$, are the two most prominent in the literature. The former is the canonical basis of the space of all functions, while the latter is the canonical basis when Boolean functions are viewed as polynomial functions over $\F_2$. Their close connection with the combinatorics of the subset lattice on $2^{\IeFer n}$, together with the markedly different random models they generate, makes them natural objects of study in the present work.

Let us finally emphasize that if the coordinate vector $(\beta_{n,K})_{K\subset\IeFer n}$ is chosen uniformly at random, then the resulting distribution on functions $\phi_n\in\sE^{\sE^n}$ is uniform, irrespective of the chosen basis. Thus the choice of representation does not affect the marginal law of $\phi_n$, although it may have a profound impact on the joint structure
of the sequence $(\phi_n)_{n\ge1}$.

\textit{Note on $\max$ vs $\min$.} 
Note that the previous discussion relies on the arbitrary identification of $-1,+1$ with $1,0$, respectively. For the sake of completeness, let us mention that the other identification, i.e.\ with $0,1$ respectively, would mainly amount to changing $\min$ for $\max$ in the I.I.\ representation. The proofs adapt seamlessly to this essentially equivalent setting. 

\section{Useful lemmas}\label{appendix}

\begin{lemma}\label{lem:coupling}
Let $S_0$ be an integer-valued r.v., and $(z_n)_{n\ge1}$ be i.i.d.\ Bernoulli r.v.'s of parameter $1/2$, independent of $S_0$. Define, for all $n\ge1$, $S_n=S_0+z_1+\cdots+z_n$. 

Let $\alpha\in(0,1)$, and $\ell_n:=\floor{\frac{1-\alpha}2\log_2 n}$.  For all $n\ge1$,  one may couple $S_n$ with a random variable $U_n$ that is uniformly distributed on $\{0,1,\ldots,2^{\ell_n}-1\}$, in such a way that 
\[\P(S_n \not\equiv U_n\modulo2^{\ell_n})\le 2e^{-\pi^2 n^\alpha}.\]
\end{lemma}

\begin{proof}
The proof relies on bounding the total variation distance between the distribution of $(S_n\modulo2^{\ell_n})$ and the uniform distribution on $\{0,1,\ldots, 2^{\ell_n}-1\}$. 
For any $k\in\Z$, by the (elementary\footnote{For any r.v.\ $X$ on $\Z$, and any $k\in\Z$, $N\in\N$, 
$\sum_{j=0}^{N-1}e^{\frac{2i\pi}N j(X-k)} =N\ind{\{X-k\equiv0\modulo N\}}$ 
hence
$\P(X \equiv k\modulo N)
	 = \frac1N\sum_{j=0}^{N-1}e^{-\frac{2i\pi}Njk}\E[e^{\frac{2i\pi}NjX}]$}) inversion formula for the discrete Fourier transform one has 
\begin{align*}
\P(S_n\equiv k\modulo 2^{\ell_n})
	& = 2^{-\ell_n}\sum_{j=0}^{2^{\ell_n}-1}e^{-2i\pi jk2^{-\ell_n}}\E[e^{2i\pi j S_n 2^{-\ell_n}}]
\end{align*}
and, for any value of $j$, 
\[\E[e^{2i\pi jS_n 2^{-\ell_n}}]=\E[e^{2i\pi jS_02^{-\ell_n}}]\pa[\Big]{\frac12+\frac12\exp(2i\pi j2^{-\ell_n})}^n = \E[e^{2i\pi jS_02^{-\ell_n}}]\cos(\pi j 2^{-\ell_n})^n e^{i\pi nj 2^{-\ell_n}}
\]
so that, singling out the term $j=0$, 
\begin{align*}
\abs[\big]{\P(S_n\equiv k\modulo 2^{\ell_n})-2^{-\ell_n}}
	& \le 2^{-\ell_n}\sum_{j=1}^{2^{\ell_n}-1}\abs[\Big]{\cos\pa[\Big]{\pi j2^{-\ell_n}}}^n\\
	& = 2^{1-\ell_n}\sum_{j=1}^{2^{\ell_n-1}-1}\cos\pa[\Big]{\pi j2^{-\ell_n}}^n\\
\end{align*}
by symmetry of $\cos$ aroung $\frac\pi2$ and removing the middle term $\cos\frac\pi2=0$. 
Using the inequality $\cos x\le e^{-\frac{x^2}2}$, holding for all $x\in[0,\pi/2]$ (as can be seen for instance by integrating twice the inequality $\phi''\ge1$ where $\phi=\log\cos$, on $[0,\pi/2)$), and then the definition of $\ell_n$, one can bound, for $n\ge1$, 
\begin{align*}
\sum_{j=1}^{2^{\ell_n-1}-1}\cos\pa[\Big]{\pi j2^{-\ell_n}}^n
	& \le \sum_{j=1}^{2^{\ell_n-1}-1}e^{-n(\pi^2/2) j^2  4^{-\ell_n}}
	 \le \sum_{j=1}^{\infty}e^{-n(\pi^2/2) j 4^{-\ell_n}} \\
	& \le \sum_{j=1}^\infty e^{-j (\pi^2/2) n^{\alpha}}  =\frac1{1-e^{-(\pi^2/2)n^{\alpha}}} e^{-(\pi^2/2)n^{\alpha}}
\end{align*}
hence finally, for all $n\ge1$, 
\begin{align*}
\frac12\sum_{k=0}^{2^{\ell_n}-1}\abs[\big]{\P(S_n\equiv k\modulo 2^{\ell_n})-2^{-\ell_n}}\le \frac1{1-e^{-\pi^2n^{\alpha}}}e^{-(\pi^2/2)n^{\alpha}}\le2e^{-(\pi^2/2)n^{\alpha}}. 
\end{align*}
The above left hand side is the total variation distance between the law of $(S_n\modulo2^{\ell_n})$ and the uniform distribution on $\{0,1,\ldots,2^{\ell_n}-1\}$. Classically (see for instance~\cite[p.50]{LevinPeres}), this implies the existence of a coupling between $S_n$ and a uniform random variable $U_n$ such that the probability $\P(S_n\not\equiv U_n\modulo 2^{\ell_n})$ is equal to this distance, hence bounded by $2e^{-(\pi^2/2)n^\alpha}$. 
\end{proof}

\begin{lemma}\label{lem:subexponential}
    For any $a\in[0,1)$, as $n\to\infty$, \[\sum_{k=0}^n {n\choose k} a^{2^k} = e^{O(\log^2 n)}.\]
\end{lemma}

\begin{proof}
Let us denote by $v_n$ the above sum. We bound ${n\choose k}\le n^k$ and split the sum using $L_n=\floor{(1+\eps)\log_2 n}$ for an arbitrary $\eps>0$, to obtain
\begin{align*}
v_n 
	& \le \sum_{k=0}^{L_n-1}{n}^ka^{2^k}+
	 \sum_{k=L_n}^{n}{ n}^ka^{2^k}\\
	& \le L_n {n}^{L_n} + n{n}^na^{2^{L_n}}.
\end{align*}
Now, $n^{1+\eps}/2<2^{L_n}\le n^{1+\eps}$ from which we deduce that
$$\log(L_n {n}^{L_n}) \le \log(L_n)+(1+\eps)\log_2 n\log n = O((\log n)^2)$$
and
$$\log(nn^na^{2^{L_n}})  <
\log n+n\log n-\frac12n^{1+\eps}\log(1/a)\to-\infty,$$
so that the first term in the bound of $v_n$ is of the desired order, and the second one vanishes. 
\end{proof}

NB. The optimal bound would be $e^{O(\log n\cdot\log\log n)}$, however the above bound (or even $e^{o(n)}$) suffices in our applications.

\subsection*{Acknowledgments} This research was supported by the Australian Research Council (ARC) grant DP230102209.
L.T.\ wishes to thank the Center for Modeling of Stochastic Systems at Monash University, where this project was initiated, for its hospitality.

\end{document}